\documentclass{amsart}
\usepackage{amsmath,amsfonts,amssymb,color}
\usepackage{a4wide,amsthm}
\usepackage[english]{babel}
\usepackage{graphicx}
\usepackage{hyperref,pdfsync}

\renewcommand{\div}{\textrm{div}\,}

\newcommand{\BFG}{$\text{(BF)}$}

\newtheorem{theorem}{Theorem}
\newtheorem{corollary}[theorem]{Corollary}
\newtheorem{proposition}[theorem]{Proposition}
\newtheorem{lemma}[theorem]{Lemma}
\newtheorem{definition}[theorem]{Definition}
\newtheorem{rem}{Remark}

\begin{document}

\title{Existence of global strong  solutions to a beam-fluid interaction system}

\author{C. Grandmont$^*$}

\thanks{\hskip -0.5cm \noindent$^*$ Email: celine.grandmont@inria.fr\\[2pt]
INRIA Paris-Rocquencourt, Project-team REO,\\ 
Building 16, BP 105,\\ 
78153 Le Chesnay cedex, France\\[6pt]}
\author{M. Hillairet$^\dag$}\thanks{\noindent$^\dag$ Email: matthieu.hillairet@univ-montp2.fr\\[2pt]
Universit\'e Montpellier II,
France}

\date{\today}

\maketitle

\begin{abstract}
We study an unsteady non linear fluid-structure interaction problem which is a simplified model to describe blood flow through viscoleastic arteries. We consider a Newtonian incompressible two-dimensional flow described by the Navier-Stokes equations set in an unknown domain depending on the displacement of a structure, which itself satisfies a linear  viscoelastic beam equation. The fluid and the structure are fully coupled via  interface conditions prescribing the continuity of the velocities at the fluid-structure interface and the action-reaction principle. We prove that strong solutions to this problem are global-in-time. We obtain
in particular that contact between the viscoleastic wall and the bottom of the fluid cavity does not occur in finite time. 
To our knowledge, this is the first occurrence of a no-contact result, but also of existence of strong solutions globally in time, in the frame of interactions between a viscous
fluid and a deformable structure. 
\end{abstract}

\bigskip
\bigskip

\section{Introduction}

In this paper, we focus on the interactions between a viscous incompressible Newtonian fluid and a moving viscoelastic structure located on one part of the fluid domain boundary. This work is motivated by the study of blood flow in arteries and the fluid-structure interaction (FSI) model we consider here can be viewed as a simplified version of a standard model/benchmark  for FSI problems/solvers in hemodynamics \cite{QTV00}, \cite{FGNQ01}. Here we are interested in global existence result and the possibility of collapse of the arterial wall. Consequenlty, we investigate whether or not collision occurs between the moving boundary and the bottom of the fluid cavity together with the existence of global-in-time strong solutions. 

\medskip

A vast majority of studies on the existence of solution for fluid-structure interaction problems concerns the motion of a rigid solid in a viscous incompressible Newtonian fluid, whose behavior is then described by the Navier--Stokes equations (see \cite{Se87,Yu,Ho-St99,De-Es99,CSMT00,De-Es00,GM00,HOS1,SM-St-T02,Feireisl03,Ta03,Ta-Tu04}). A challenging point is the possibility of body-body or body-boundary collisions. In particular, these existence results are valid up to contact, except  those of \cite{SM-St-T02} and \cite{Feireisl03} where special weak solutions after possible collisions are built  in the 2D case and in the 3D case respectively.  The contact issue is further investigated in \cite{HES} and \cite{Hil07} where a no-collision result is proven in a bounded two-dimensional cavity.  The three-dimensional situation is also explored in \cite{HT09} and in \cite {HT10}. 
We mention that, in \cite{Sta07,SM-St-T02}, collisions, if any, are proved to occur with zero relative {\it translational} velocity as soon as the boundaries of
the rigid bodies are smooth enough and the gradient of the surrounding fluid velocity field is square integrable.  A complementary study of 
the influence of the smoothness of boundaries on the existence of collisions has been recently tackled in \cite{GVH10} and \cite {GVH12}. 
In  \cite{GVH10} it is proved that below a critical regularity of a 2D  rigid body boundary, namely $C^{1, 1/2}$, collisions may occur ; 
in \cite{GVH} and \cite{GVHW} slip boundary conditions  at the fluid-structure interface are introduced and existence of weak solutions up to collision/existence of finite time collisions are proven respectively. 

Fewer studies consider the case of an elastic structure evolving in a viscous incompressible Newtonian  flow. 
We refer the reader  to \cite{DEGLT01} and \cite{BST12} where the structure is described by a finite number of eigenmodes or to \cite{Bou07} for an artificially damped elastic structure while, for the case of a three-dimensional elastic structure interacting with a three-dimensional fluid, we mention \cite{G02,Ga-K09} in the steady state case and \cite{Cou-Shk05,Cou-Shk06,Kuk-Tu12,RV14} for the full unsteady case. In the latter, the authors consider the existence of strong solutions for small enough data locally in time.  In this field, the question of selfcontact and/or body-boundary collisions remains entirely open to our knowledge.

\medskip

 Concerning the fluid-beam, or more generally fluid-shell coupled systems, that we consider herein, the 2D/1D steady state case is considered in \cite{G98} for homogeneous Dirichlet boundary conditions on the fluid  boundaries (that are not the fluid-structure interface). Existence of a unique  strong enough solution is obtained for small enough applied forces.  In  the unsteady framework  we refer to \cite{CDEG05} where a 3D/2D fluid-plate coupled system is studied and where the structure is a damped plate in flexion. The case of an undamped plate is studied in \cite{G08}. The previous results deal with the existence of weak solutions, {\em i.e.} in the energy  spaces, and rely on the only tranversal motion of the elastic beam that enables to circumvent the lack of regularity of the fluid domain boundary (that is not even Lipschitz). These results also apply to a 2D/1D fluid-shell coupled problem which is recently considered in  \cite{MC13}. In this reference, the authors give an alternative proof of existence of  weak solutions based on ideas coming from numerical schemes \cite{GGCC09}. The existence of strong solutions for 3D/2D, or 2D/1D coupled problem involving a damped elastic structure is studied in \cite{Vei04,Lequeurre11,Lequeurre13}. The proofs  of  \cite{Lequeurre11,Lequeurre13} are based on a splitting strategy for the Stokes system and on an implicit treatment of the {\it so called} fluid added mass effect. Finally, the coupling of a 3D Newtonian fluid and a linearly elastic Koiter shell is recently studied in \cite{Len-Ru}. In this study, the mid-surface of the structure is not flat anymore and existence of weak solutions is obtained.

\medskip

In all the results we mentioned up to now, the existence of strong solutions is obtained locally in time while existence of weak solutions is obtained up to contact between the elastic structure and the fluid cavity. Consequently, if one wants to prove existence of solutions globally in time, either one should give a sense to solutions in case of collision or one should prove that no collision occurs. In this paper we investigate these issues on a  2D/1D fluid-beam coupled problem in which the beam is viscoelastic and moves only in the tranversal direction relatively to its flat mid-surface.  One already knows that a unique strong solution exists locally in time \cite{Lequeurre11}, whereas weak solutions exist as long as the beam does not touch the bottom of the fluid cavity (see \cite{CDEG05} or \cite{G08} and  \cite{MC13} in the undamped beam case).  Note that, in this model, the fluid domain is a subgraph whose regularity depends on the structure displacement. In connexion with the rigid-body case, we mention that energy estimates ensure that the beam displacement belongs to $L_t^\infty(H_x^2)$ that embeds in $L_t^\infty( C^{1, \alpha}_x)$, for $\alpha =1/2$. This corresponds precisely to the threshold exhibited in \cite{GVH10}.  In this paper, the strategy we develop is first to  prove that no contact occurs and, next to propagate  the solution regularity. In the second step, the cornerstone is an elliptic regularity result  for the inhomogeneous Stokes system valid for nonstandard regularity of the domain boundary.

\medskip

\section{General setting, main result and formal argument}
We consider a 2D container whose boundary is made of a 1D viscoelastic beam, which is a simplified linear viscolelastic Koiter shell model \cite{MC13}.  Due to the complexity of the underlying fluid-structure interaction problem we assume that only the upper part of the fluid cavity is moving.
The fluid domain denoted $\mathcal {F}(t)\subset \mathbb{R}^2$ depends then on time since it depends on the structure displacement $\eta$. It reads
$$\mathcal F(t) := \{(x,y)\in {\mathbb R}^2\,, \: x \in \mathbb (0,L)\,, \: y \in (0,h(x, t))\}\,.
$$
where $(x, t) \mapsto h(x, t)=1+\eta (x, t)$ stands for the ``deformation" of the beam. 
We assume that the fluid is two dimensional, homogeneous, viscous, incompressible and Newtonian. Its velocity-field $u$ and internal pressure $p$
satisfy the incompressible Navier--Stokes equations in $\mathcal F(t)$:
\begin{eqnarray}\label{eq_NS}
\rho_f( \partial_t u + u \cdot \nabla u )- \div \sigma(u,p) &=&0\,,  \\ 
	\label{eq_incomp}					\div u &=& 0\,.
\end{eqnarray}
The fluid stress tensor $\sigma(u,p)$ is given by the Newton law:
$$
\sigma(u,p) = \mu (\nabla u + \nabla u^{\top}) - p \text{I}_2\,. 
$$
Here $\mu$ denotes the viscosity of the fluid and $\rho_f$ its density.
The  structure motion  is given by a linear damped beam equation in flexion:
\begin{eqnarray} \label{eq_elasticite}
\rho_s\partial_{tt}h  - \beta \partial_{xx}h  + \alpha \, \partial_{xxxx}h - \gamma \, \partial_{xxt} h &=& \phi(u,p,h), \quad \mbox{on } (0,L)\,,
\end{eqnarray} 
where $\alpha,\beta,\gamma$ are non-negative given constants and $\rho_s>0$ denotes the structure density. 

\medskip

We emphasize that the beam equation is  set in  a reference configuration whereas the fluid equations
are written in Eulerian coordinates and consequently in an unknown domain. Note moreover that we choose to write the beam equation on  $h$  and not on the beam displacement $\eta=h - 1$ as is standard, since this is equivalent in the case one considers here and since it simplifies the presentation.
The fluid and beam equations are coupled through the source term $\phi(u,p,h)$ in \eqref{eq_elasticite}, which corresponds to the trace of the second component of 
$\sigma(u,p)n \text{d$l$}$ transported in the beam reference configuration.  The coupling term writes:
\begin{equation} \label{eq_phi}
\phi(u,p,h)(x, t) = - e_2 \cdot \sigma(u,p)( x, h(x, t), t) (-\partial_x h(x, t) \, e_1 + e_2) \,,  \quad (x, t)\in (0, L)\times (0, T), 
\end{equation}
where $(e_1, e_2)$ denotes the canonical basis of $\mathbb{R}^2$.
The fluid and the beam are coupled also through the kinematic condition
\begin{equation}
u( x,h(x,t), t) = \partial_t h(x,t) e_2, \, \quad (x, t)\in (0, L)\times (0, T). \label{eq_bc2}
\end{equation}
We complement our system with the following conditions on the remaining boundaries of the container:
\begin{itemize}
\item $L$-periodicity w.r.t. $x$ for the fluid and the beam;
\item no-slip boundary conditions on the bottom of the fluid container:
\begin{equation}
u(x,0,t)  =   0\,.  \label{eq_bc1}\\
\end{equation}
\end{itemize}
Note that since the question of contact is mostly a local one we assume periodic boundary conditions at the oulet for the fluid and replace the standard assumptions of ``clamped" arterial wall also by periodic boundary conditions for the beam.
\medskip

An important remark on this coupled system  is that the incompressibility condition together with  boundary conditions imply:
\begin{equation} \label{eq_contrainte_dth}
\int_{0}^{L} \partial_t h(x,t) {\rm d}x = 0\,, \quad \forall \, t >0\,.
\end{equation}
Consequently, for any classical solution $(u,h,p)$ to this system, the right-hand side of \eqref{eq_elasticite}  must have  zero mean:
$$
\int_0^L\phi(u,p,h){\rm d}x =0.
 $$
It can be achieved thanks to a good choice of the  constant normalizing the pressure which is consequently uniquely defined.  The  pressure can be then decomposed,  for instance, as 
\begin{equation}
\label{decomp:pression}
p=p_0+c,
\end{equation} where, one chooses  to impose
\begin{equation} \label{eq_conditionp0}
\int_{{\mathcal F}(t)} p_0(x, y) {\rm d}x {\rm d}y =0,
\end{equation}
and where $c$ satisfies
\begin{equation}
\label{const:pression}
c(t)=\frac{1}{L}\int_0^L e_2\cdot  (\sigma(u, p_0))(x, h(x,t), t) (-\partial_x h(x, t) \, e_1 + e_2){\rm d}x. 
\end{equation}
This constant $c$ is the Lagrange multiplier associated with the constraint \eqref{eq_contrainte_dth}.
Another mathematical way to compel the solution with this compatibility condition, without defining the physical average of the fluid pressure $c$, is to introduce the  $L^2$-projection operator on the set of $L$-periodic functions whose averages are equal to zero on (0, L), denoted  $M_s$,  and rewrite \eqref{eq_elasticite} as
\begin{equation} \label{eq_elasticite'}
\rho_s\partial_{tt} h - \beta \partial_{xx} h + \alpha  \partial_{xxxx}h - \gamma \partial_{txx} h = M_s \phi(u,p,h)\,.
\end{equation}
This is the choice made in \cite{Lequeurre11}.

\subsection{Main result}

In what follows, we call \BFG\,  the ``beam-fluid" system \eqref{eq_NS}-\eqref{eq_incomp}-\eqref{eq_elasticite}-\eqref{eq_phi}-\eqref{eq_bc2}-\eqref{eq_bc1}-\eqref{decomp:pression}-\eqref{eq_conditionp0}-\eqref{const:pression}. We study herein the \BFG\,   system, completed 
with initial conditions:
\begin{eqnarray}
h(x,0) &=& h^0(x) \,,   \qquad x \in (0, L), \,  \label{eq_ic1}\\
\partial_th(x, 0) &=& \dot{h}^0(x) \,, \qquad x \in (0, L) \,,\label{eq_ic2} \\
u(x,y,0) &=& u^0(x,y) \,, \; \;(x,y) \in \{ x \in (0, L)\,,  y \in (0,h_0(x))  \}  =: \mathcal{F}^0\,.\label{eq_ic3} 
\end{eqnarray}
Our aim is to study this Cauchy problem and, specifically, to prove the existence of  a unique global-in-time strong solution. 

\medskip

We first give some notations and definitions and make precise the functional framework.  For any given  non-negative function ${b} \in C_{\sharp}(0,L)$, {\em i.e.} the set of  continuous and $L$-periodic functions on $\mathbb R$, we define  
$$
\Omega_{b} := \{(x,y) \in \mathbb R^2 \,,\text{such that }  x\in (0, L),  y \in (0, {b}(x))\}\,.
$$ 
With this definition $ \mathcal{F}(t)=\Omega_{h(t,\cdot)}$.
We denote by ${C}^\infty_\sharp(\Omega_{b})$  the restriction on $\Omega_b$  of  $L$-periodic functions in $x$ indefinitely differentiable on 
$$
{\mathcal{ O}}_{b} =\{(x,y) \in \mathbb R^2 \,,\text{ s.t. },  y \in (0, {b}(x))\}
$$
Note that ${\mathcal {O}}_{b} +L e_1\subset {\mathcal{ O}}_{b}$ and ${\mathcal {O}}_{b}= \cup_{k\in {\mathbb Z}}\Omega_{b} +Lk e_1$. We introduce the classical spaces  $L^p_{\sharp}(\Omega_{b}) $ and $H^m_{\sharp}(\Omega_{b}) $ respectively  as the closures  of  ${C}^\infty_\sharp(\Omega_{b})$ in $L^p(\Omega_b)$ or $H^m(\Omega_b)$.  We define in the same way  $C^s_{\sharp}(0, L),$ $L^p_\sharp(0,L)$ and $H^m_\sharp(0,L)$.
 More generally, the subscript $\sharp$ stands for the periodic version in the first variable of a function space.
 We emphasize that contrary to the usual convention, we consider that time is the last variable of a function. This enables to write a unified definition for periodic functions whether they depend on one space variable only (such as the height function $h$) or two space variables (such as the velocity-field $u$). 
 Finally, we set:
$$
L^{2}_{\sharp,0} (0, L):= \left\{ f \in L^2_{\sharp}(0, L) \text{ s.t.} \int_{0}^{L} f(x){\rm d}x=0\right\}\,,
$$
and, in the same way,
$$
L^{2}_{\sharp,0} (\Omega_{b}):= \left\{ f \in L^2_{\sharp}(\Omega_{b}) \text{ s.t.} \int_{\Omega_{b}} f(x){\rm d}x=0\right\}\,,
$$
Then the projection operator  $M_s$, that is applied to equation \eqref{eq_elasticite} leading to equation \eqref{eq_elasticite'},
is the orthogonal-projector from $L^2_{\sharp}(0, L)$ onto $L^2_{\sharp,0}(0, L).$

 \medskip

We state our main result as follows
\begin{theorem} \label{thm_main}
Let us consider $\alpha>0$,  $\beta\geq 0$ and  $\gamma>0$.
Assume that the initial data $(h^0,\dot{h}^0,u^0)$  satisfy:
\begin{itemize}
\item $(h^0,\dot{h}^0) \in H_\sharp^3(0, L) \times H_\sharp^1(0, L)$, 
\item $u^0 \in H_\sharp^1(\mathcal F^0)\,,$
\item no-slip condition is fulfilled initially: 
\begin{equation}
u^0(x,0) = 0\,, \quad u^0(x,h^0(x)) = \dot{h}^0(x) \,, \quad \forall \, x \in (0, L)\,,
\end{equation}
\item no-contact and incompressibility compatibility conditions are fulfilled initially:
\begin{eqnarray} \label{eq_comp}
&& \min_{x \in \mathbb [0, L]} h^0(x) >0 \,, \quad \int_{0}^{L} \dot{h}^0(x) {\rm d}x = 0\,, \\
&& {\rm div} u^0 = 0 \quad \text{ on $\mathcal F^0$}\,.
\end{eqnarray}
\end{itemize}
Then \BFG\  has a unique global-in-time strong solution.
\end{theorem} 

A precise definition of what is a ``strong solution" is given in 
{Section \ref{sec_loc}}.  Our proof for {Theorem \ref{thm_main}} follows a classical scheme: 
local-in-time existence and uniqueness of solutions, blow-up alternative and {\it a priori} estimates.  
Local-in-time existence and uniqueness  of strong solutions has already been tackled in \cite{Lequeurre11} for clamped boundary conditions instead of periodic boundary conditions. In {Section \ref{sec_loc}}, we explain shortly how this result can be adapted to our functional framework with periodic lateral boundary conditions. 
 This construction leads to  the existence of a unique maximal solution for any given initial data, that blows up in finite time if and only if
the quantity 
\begin{multline} \label{eq_REG}
\mathcal C(t) := \displaystyle{\sup_{x \in [0, L]}} \dfrac{1}{h(x,t)} +  \int_{0}^{L} \left( \alpha |\partial_{xxx} h (x,t)|^2 + \gamma |\partial_{tx} h(x,t)|^2 \right) {\rm d}x  \\ +  \int_{0}^{L}\int_0^{h(x,t)} \mu|\nabla u(x,y,t)|^2 {\rm d}x{\rm d}y 
 \end{multline}
 diverges in finite time (see Corollary \ref{cor_blowup}). Note that existence of weak solutions as long as the beam does not touch the bottom of the fluid cavity can be obtained also by adapting \cite{CDEG05,G08} or \cite{MC13} to our setting.

\medskip

In this paper, the main novelty  is  the computation, for any local-in-time strong solutions to \BFG, of a new {\em a priori} estimate on $\mathcal C,$  defined by \eqref{eq_REG}.  This estimate enables us to derive a regularity estimate valid on any given time interval $(0, T)$.
We emphasize that, in order to obtain these estimates, we have to assume that $\alpha >0$ and  $\gamma >0$.  This ensures first that the elastic boundary remains regular, second that the beam dissipates energy. Whether or not these  assumptions can be dropped remains a completly open question (notice that existence of weak solutions before contact is still valid for $\alpha=\gamma=0$ and $\beta>0$).

\subsection{Formal argument}
Before considering the full Navier--Stokes/beam system, and in order to illustrate the different steps of the proof,  let us first consider a reduced model for which we derive similar bounds (at a formal level for conciseness). This coupled system writes
\begin{align}
\rho_s\partial_{tt}{b}  - \beta \partial_{xx}{b}  + \alpha \, \partial_{xxxx}{b}  - \gamma \, \partial_{xxt}  {b} =& \; q ,& \mbox{on } (0, L),  \label{eq_rho1}\\
\partial_{t} {b} =& \; \partial_x [{b}^3 \partial_{x} q], &\mbox{on } (0, L) \,, \label{eq_rho2}
\end{align}
where ${b}$ stands for the deformation of the beam (analogue of $h$) and $q$ denotes the fluid pressure.
To complete these equations, we require that  $q$ satisfies 
\begin{equation} \label{eq_rho3}
\int_{0}^{L} q(x,t) {\rm d}x = 0 \quad \forall \, t >0\,,
\end{equation}
and we impose periodic boundary conditions in $x$. We note that this system is related to \BFG\, as it can be obtained formally  by setting $h = \varepsilon {b} $ and letting $\varepsilon$ go to zero. The second equation \eqref{eq_rho2},
\emph{i.e.} the Reynolds equation, is a reduced model for the Navier--Stokes equations (completed with no-slip boundary conditions) 
in the thin film (or lubrication) regime. We refer  to \cite[Section 5.B]{Lealbook} for a detailed derivation of the Reynolds  equation.
 
\medskip

Let $({b},q)$ be a (classical) $L$-periodic solution to \eqref{eq_rho1}--\eqref{eq_rho3} on $(0,T)$ with $T>0.$ First, multiplying \eqref{eq_rho1} by $\partial_t \rho$ and combining with \eqref{eq_rho2} multiplied by $q$ yields:
$$
\dfrac{1}{2} \dfrac{\textrm{d}}{\textrm{d}t} \left[ \int_{0}^{L}  \left(\rho_s|\partial_t{b}|^2  +\beta |\partial_{x}{b}|^2 + \alpha |\partial_{xx} {b}|^2\right) \right] + \int_{0}^{L} \left(\gamma |\partial_{tx}{b}|^2 + {b}^3  |\partial_x q|^2 \right)= 0\,.
$$
We obtain that there exists a constant $C_0$ depending only on initial data for which:
\begin{multline} \label{reg_rho1}
\sup_{t  \in (0,T)} \left(\rho_s\| \partial_t {b} \, ; \, L_\sharp^2(0,L)\|^2 + \alpha\|{b}\, ; \, H_\sharp^2(0, L)\|^2 + \beta\|{b}\, ; \, H_\sharp^1(0, L)\|^2 \right)  \\
+ \int_0^T \left(\gamma\|\partial_{t}{b} \, ; \, H_\sharp^1(0, L) \|^2  +  \| {b}^{\frac 32}\partial_x q  \, ; \, L_\sharp^2(0, L)\|^2 \right)  \leq C_0\,.
\end{multline}
In what follows, $C_0$ denotes a constant depending only on the initial data but which may  vary between lines.
To obtain a lower bound on ${b}$, we multiply \eqref{eq_rho1} by $-\partial_{xx}{b}$ and integrate over $(0,L)$.  We then integrate by parts  in space.  By taking into account the periodic boundary conditions and 
\eqref{eq_rho2}, we obtain:
\begin{eqnarray*}
&\displaystyle\dfrac{{\rm d}}{{\rm d}t} \left[ \int_{0}^{L} \left( \frac{\gamma}{2} |\partial_{xx} {b}|^2 - \rho_s \partial_{t}{b} \, \partial_{xx} {b} \right) \right] 
+ \int_{0}^{L} \left( \beta |\partial_{xx}{b}|^2 + \alpha |\partial_{xxx}{b}|^2 - \rho_s  |\partial_{tx}{b}|^2\right) \\
 & = - \displaystyle\int_{0}^{L} q \partial_{xx}{b} = \int_{0}^{L} \partial_x q \partial_{x} {b} = \int_{0}^{L}{b}^3 \partial_x q \  (\frac{1}{{b}^3}\partial_{x}{b})  \\
&  = - \displaystyle\dfrac{1}{2}\int_{0}^{L} {b}^3 \partial_x q \   \partial_{x}\left[ \dfrac{1}{{b}^2}\right] 
				= \dfrac{1}{2} \int_{0}^{L} \dfrac{\partial_t{b} }{{b}^2}\,.
\end{eqnarray*} 
Finally, we deduce:
\begin{equation} \label{est_dist}
\dfrac{{\rm d}}{{\rm d}t} \left[ \int_{0}^{L} \frac{1}{2} \left( \gamma|\partial_{xx}{b}|^2 +  \frac{1}{{b}}\right)  - \int_{0}^{L}  \rho_s\partial_{t} {b} \, \partial_{xx}{b}  \right]  
+ \int_{0}^{L} \left( \beta |\partial_{xx}{b}|^2 + \alpha |\partial_{xxx} {b}|^2 \right) = \int_{0}^{L} \rho_s|\partial_{tx}{b}|^2\,.
\end{equation}
Combining with the previous bound \eqref{reg_rho1}, we get that there exists  a constant $C_0$ such that:
\begin{equation} \label{reg_rho2}
\sup_{t  \in (0,T)} \left(\gamma\|{b} \, ; \, H_\sharp^2(0,L)\|^2 + \|{b}^{-1} \, ; \, L_\sharp^1(0, L)\| \right) + \int_0^T \alpha\|{b} \, ; \, H_\sharp^3(0, L) \|^2 +\beta\|{b} \, ; \, H_\sharp^2(0, L) \|^2 \leq C_0\,.
\end{equation}
Note that, if $\rho_s\neq 0$, to obtain the previous estimate one should compute an upper bound for
$$
\int_0^T\int_{0}^{L} \rho_s|\partial_{tx} {b}|^2\,, \quad \sup_{t\in (0,T)}\int_{0}^{L}  \rho_s\partial_{t} {b} \, \partial_{xx}{b}.
$$
From \eqref{reg_rho1} these terms are bounded if $\gamma >0$ and if $\gamma$ or $\alpha$ is strictly positive respectively.   
At this point, we call a real-analysis lemma which states that there exists a continuous function $D_{min}$ for which:
$$
 \|{b}^{-1} \, ; \, L_\sharp^{\infty}(0, L)\| \leq D_{min} (\|{b} \, ; \, H_\sharp^2(0, L)\|, \|{b}^{-1} \, ; \, L_\sharp^1(0, L)\|)
$$
(see Appendix \ref{app_inequalities} for a proof). The first consequence of this inequality is that \eqref{reg_rho2} implies that ${b}$ remains away from zero
uniformly in time.  Combining this information with the dissipation estimate \eqref{reg_rho1} we obtain that there exists a constant $C_0$ for which: 
$$
\sup_{t \in (0,T)} \left( \|{b}^{-1}(\cdot,t) ; L_\sharp^{\infty}(0, L)\|\right) + \int_0^T \|q ; H_\sharp^1(0, L) \|^2 \leq C_0\,.
$$

Consequently the pressure is bounded in $L^2(0, T; H_\sharp^1(0, L))$.  Next we derive a regularity estimate  for the deformation ${b}$ by multiplying \eqref{eq_rho1} by $-\partial_{txx}{b}.$ This yields after integration by parts in space:
$$
\dfrac{1}{2}\dfrac{\rm d}{{\rm d} t} \left[ \int_{0}^{L} \rho_s |\partial_{tx}{b}|^2 + \alpha |\partial_{xxx} {b}|^2 + \beta |\partial_{xx}{b}|^2 \right]
+ \gamma \int_{0}^{L} |\partial_{txx}{b}|^2 = \int_{0}^{L} q \partial_{txx} {b}.
$$
Thanks to the $L^2(0, T; H_\sharp^1(0, L))$-bound on $q$, we reach the required estimate that  enables us to extend solutions globally in time:
\begin{multline} \label{reg_rho3}
\sup_{t \in (0,T)} \left(\alpha\|{b} ; H_\sharp^3(0, L)\|^2 +\beta\|{b} ; H_\sharp^2(0, L)\|^2 + \rho_s\|\partial_t {b} ; H_\sharp^1(0, L)\|^2  \right)  \\
+ \gamma\int_0^T  \|\partial_{t} {b} ; H_\sharp^2(0, L)\|^2 \leq C_0\,.
\end{multline}
This ends the formal proof of a no collision result, on the one hand, and of a global-in-time existence of strong solutions, on the other hand.

\medskip

In {Section \ref{sec_glob}}, we prove that comparable estimates hold true for the complete coupled system \BFG. First, considering the full beam/Navier--Stokes system, the analogue of \eqref{reg_rho1}  corresponds to the (already-known) classical decay of kinetic energy. To obtain a similar estimate to \eqref{reg_rho2}, we multiply \eqref{eq_elasticite'} by  $-\partial_{xx} h$ and   multiply \eqref{eq_NS} by a suitable extension of $-\partial_{xx} h$.  The choice of this extension is a key point of the proof (see Section \ref{sec_distance}). We  then obtain an identity similar to \eqref{est_dist} with additional remainder terms that we bound thanks to the energy estimate and for which a control of $h$  in $L^2(0, T; H_\sharp^3(0, L))\cap L^\infty(0, T ; H_\sharp^2(0, L))$, resp. a control of $\partial_th$ in $L^2(0, T; H_\sharp^1(0, L))$  is needed (hence $\alpha >0$, resp. $\gamma>0$). 
The extension of the last estimate \eqref{reg_rho3} is more involved. Indeed, when dealing with the full Navier--Stokes/beam system, we also have to control the fluid velocity-field  $u$ in $L^\infty(0, T; H_\sharp^1(\mathcal F(t)))$ (and not only the pressure field as for the toy model).
When working in cylindrical domains, such an estimate  is obtained  by multiplying \eqref{eq_NS} with $\partial_t u$ and  by applying elliptic estimates for the Stokes system  in order to  bound  the convective terms.
However,  these elliptic estimates are classically proven in $C^{1,1}$-domains or $W^{2,\infty}$-domains \cite{Bello,Galdi}. Here as $h$ is merely  $L^\infty(0, T; H_\sharp^2(0, L))$, we cannot directly apply  these standard regularity properties. So, we need to extend the elliptic results  for the Stokes system to domains which are only subgraphs of $H^2$-functions and analyze precisely the dependency of the associated  elliptic estimates with respect to the norms of $h$ (see Lemma \ref{lem_ellest}). This proof is an adaptation to a periodic framework of a lemma that can be found in \cite{Grandmontpp}. Moreover, as the fluid domain is moving with time, instead of   $\partial_t u$, we need to consider  a multiplier that takes into account  this motion. The most natural choice is $\partial_{t} u + u\cdot \nabla u.$
But, this  function is not divergence-free and consequently pressure terms appear that cannot be handled easily. To avoid this difficulty, we mimic the method used in \cite{Cum-Tak07}  in the framework of fluid/solid interactions. We introduce a divergence-free multiplier avoiding the introduction of the pressure in the regularity estimate. Moreorever this multiplier is chosen so that the associated multiplier for the structure equation is $\partial_{tt} h$. Nevertheless,  a special attention needs to be paid since  the structure motion is, once again,  less regular than when considering fluid/solid interactions.   

\medskip

The outline of this paper is as follows. In next section, we focus on the change of variables turning the beam/fluid system into a quasilinear system in a fixed geometry. 
We recall the construction of local-in-time strong solutions of  \cite{Lequeurre11}  and adapt this result to our  periodic boundary conditions framework. 
We end the section by a technical proposition dealing with elliptic estimates for the inhomogeneous Stokes system in $H^2$- subgraph domains. 
The third and last sections are devoted to the proof of {Theorem \ref{thm_main}}. They are divided into three subsections corresponding respectively to the extension of the three estimates
\eqref{reg_rho1}, \eqref{reg_rho2} and \eqref{reg_rho3} to solutions  of the coupled problem \BFG.

\section{Local-in-time strong solutions and technical lemmas.} \label{sec_loc}

In this section,  we first adapt the construction of local-in-time strong solutions of  \cite{Lequeurre11} to our periodic setting.  To this end, we will apply the following change of variables:
\begin{equation} \label{def_hat}
\hat{f}(x,z) = f(x,h(x) z )\,, \quad \forall \, (x,z) \in \Omega_1\,.
\end{equation}
To measure the regularity of such a change of variable, the following technical proposition is required:
\begin{proposition} \label{prop_cdv}
Let us consider $h \in H^2_{\sharp}(0, L)$ satisfying $\min_{x \in [0, L]} h(x) >0.$ Then for any given $m \leq 2,$
\begin{itemize}
\item the mapping  $f \mapsto \hat{f}$ defined by \eqref{def_hat} realizes a linear homeomorphism from $H^m_{\sharp}(\Omega_{h})$ onto  $H^m_{\sharp}(\Omega_1)\,,$  

\item there exists a non decreasing function $K^e_m : [0,+\infty[ \rightarrow (0,\infty)$ such that, if we assume moreover that $\|h ; H^2_{\sharp}(0, L)\| + \|h^{-1};L^{\infty}_{\sharp}(0, L)\| \leq R_0$ then there holds:
$$
\| \hat{f} ; H^m_{\sharp}(\Omega_1)\| \leq K^e_m(R_0) \|f ; H^m_{\sharp}(\Omega_{h})\| \,, \quad 
\| {f} ; H^m_{\sharp}(\Omega_h)\| \leq K^e_m(R_0) \|\hat{f} ; H^m_{\sharp}(\Omega_{1})\|\,. 
$$

\end{itemize}
 \end{proposition} 
 \begin{proof}
 The proof is standard. For $m \in \{0,1\}$ the result easily derives from the fact that $h \in W^{1,\infty}(0,L)$ is bounded from below by a strictly positive constant. 
 For $m=2,$ the key point is that the motion of the upper boundary is tranverse only so that we combine the regularity of $h$ with the following tensorization of the space $H_{\sharp}^1(\Omega_1):$
$$
H_{\sharp}^1(\Omega_1) = H_{\sharp}^1((0, L)\times(0,1))=L_{\sharp}^2(0, L; H^1(0, 1))\cap H_{\sharp}^1(0, L; L^2(0, 1)).$$
The most delicate point enters the computation of  $\|\partial_{xx} \hat{f} ; L^2_{\sharp}(\Omega_1)\|.$ We have: 
$$
\partial_{xx} \hat{f} = \widehat{\partial_{xx} f }+ 2h'z \widehat{\partial_{xy} f }+ h^{''} z \widehat{\partial_{y} f }+ (h'z)^2 \widehat{\partial_{yy}f}\,,
$$
in which the worst term is  $h'' z \widehat{\partial_{y} f }.$  It is bounded in $L_{\sharp}^2(\Omega_1)$ since $h''\in L_{\sharp}^2(0, L)$ and
$$
\widehat{\partial_{y} f } \in H_{\sharp}^1(\Omega_1)\hookrightarrow L_{\sharp}^\infty(0, L; L^2(0, 1))\cap L_{\sharp}^2(0, L; L^\infty(0, 1)).
$$ 
\end{proof}

\subsection{Construction of local-in-time solutions}
As explained previously, local-in-time existence and uniqueness of strong solutions to \BFG\, are tackled in \cite{Lequeurre11} with no normalizing condition for the pressure (\eqref{decomp:pression}-\eqref{eq_conditionp0}-\eqref{const:pression} is replaced with \eqref{eq_elasticite'}),  with homogeneous Dirichlet boundary conditions for the fluid velocity on the part of the boundary that is not elastic and with ``clamped" boundary conditions for the structure. Namely, instead of periodic boundary conditions, the displacement $\eta=h-1$ satisfies 
$$
\eta(0,t) = \eta(L,t) = \partial_x\eta(0,t) = \partial_x \eta(L,t)=0\,, \quad \forall \, t \in (0,T)\,.
$$
The proof of existence of solutions follows a classical method, also introduced  in  \cite{GM00,Ta03} when
dealing with fluid/solid interactions. To look for solutions on a time-interval $(0,T),$  new unkowns $(\hat{u},\hat{p})$ are first introduced applying the transformation \eqref{def_hat}:
\begin{equation} \label{def_up}
u(x,y,t) = \hat{u} \left(x, \dfrac{y}{h(x,t)},t\right), \quad p(x,y,t) = \hat{p}\left(x,\dfrac{y}{h(x,t)},t\right),\, \quad (x, y)\in {\mathcal F}(t).
\end{equation}
These new velocity-field and pressure $(\hat{u}, \hat{p})$ are defined in the cylindrical domain $\Omega_1 \times (0,T)$ and $(u,p,h)$ is
a solution to  \eqref{eq_NS}-\eqref{eq_incomp}-\eqref{eq_elasticite}-\eqref{eq_phi}-\eqref{eq_bc2}-\eqref{eq_bc1}-\eqref{eq_elasticite'} if and only if the triplet $(\hat{u},\hat{p},h )$ is solution to a coupled  system of quasilinear pdes that we  choose not to write here for  the sake of conciseness. The core of the existence and uniqueness result is the study of this nonlinear system. First, the author analyzes, {\em via} a semi-group approach, the resolution of the linear system obtained by linearizing around $\eta=0$ (or $h=1$), $\hat{u} = 0,\hat{p}=0.$  This study is based on an accurate treatment of the added mass effect of the fluid on the structure through an appropriate splitting of the fluid load. Then, the nonlinear terms are estimated and the author proves that they remain small for a small time. The local-in-time existence and uniqueness of a solution to the system of nonlinear pdes is finally obtained by a standard fixed point argument. 

\medskip

In our periodic framework, computation of nonlinearities might be reproduced without change while the semi-group approach might be adapted in the spirit of \cite{Lequeurre13}. 
Consequently,  for any initial data such that
\begin{equation} \label{eq_regid}
h^0 \in H^3_{\sharp}(0, L), \quad \dot{h}^0 \in H^1_{\sharp}(0, L), \quad u^0 \in H^1_{\sharp}(\mathcal F^0)\,,
\end{equation}
and satisfying  the compatibility conditions:
\begin{eqnarray} \label{eq_compic1}
&&   \min_{x \in  [0, L]} h^0(x) >0\,, \qquad \int_{0}^{L} \dot{h}^0(x) {\rm d}x = 0\,,\\
&& {\rm div}\, u^0 =0\,,   \quad \text{ on $\mathcal F^0\,,$}\label{eq_compic11}
 \end{eqnarray}
\begin{eqnarray}  \label{eq_compic2}
u^0(x,h^0(x)) &=& \dot{h}^0(x) e_2, \quad  x \in (0, L)\,,\\
u^0(x,0) &=& 0, \quad  \phantom{12443,} x \in (0, L)\,, \label{eq_compic3}
\end{eqnarray}
we obtain local-in-time existence and uniqueness of a strong solution $(\hat{u},\hat{p},h)$ to the Cauchy problem associated with the translation of
 \eqref{eq_NS}-\eqref{eq_incomp}-\eqref{eq_elasticite}-\eqref{eq_phi}-\eqref{eq_bc2}-\eqref{eq_bc1}-\eqref{eq_elasticite'} in  a fixed geometry, completed with periodic boundary conditions. 
 The solution verifies:
\begin{eqnarray} \label{reg_hatu}
&& \hat{u} \in H^1(0,T; L^2_{\sharp}(\Omega_1))  \cap C([0,T] ; H^1_{\sharp}(\Omega_1)) \cap L^2(0,T ; H^2_{\sharp}(\Omega_1))\,,
\\\label{reg_hatp}
&&
\hat{p} \in L^2(0,T ; H^1_{\sharp}(\Omega_1)),   
\\ \label{reg_h}
&& h \in H^2(0,T ; L^2_{\sharp}(0, L)) \cap L^2(0,T; H^4_{\sharp}(0, L))\,, \quad h^{-1} \in L^{\infty}((0,T) \times (0,L))
\end{eqnarray}
We emphasize that,  following the proof of \cite{Lequeurre11}, the pressure $\hat{p}$ is defined up to a constant for now. 
The regularity statement \eqref{reg_h} ensures that the function $h$ is Lipschitz on  $ [ 0,T ] \times (0,L)$. Moreover, since $h$ satisfies also 
\begin{equation} \label{hyp_h}
\dfrac{1}{h} \in W^{1, \infty}((0,T) \times (0,L)),
\end{equation}
we obtain that the domain $\mathcal{F}(t)$ and the non cylindrical domain defined by:
\begin{eqnarray*}
\mathcal{Q}_t &:= &\{ (x,y,s), \quad x \in (0, L), \quad t \in (0,t)\,,  \quad y \in (0,h(s,x)) \}\,, \quad \forall t\leq T,
\end{eqnarray*}
are both Lipschitz open subsets of $\mathbb R^2$ and $\mathbb R^3$ respectively.

\medskip

Going back to the moving domain by inverting the transformation \eqref{def_up}, we define strong  solutions  of   \BFG\  as follows
\begin{definition} \label{def_strongsolution}
Let the initial data $(h^0,\dot{h}^0,u^0) \in H^3_{\sharp}(0, L) \times H^1_{\sharp}(0, L)  \times H^1_{\sharp}(\mathcal{F}^0)$ satisfy  \eqref{eq_compic1}-\eqref{eq_compic11}-\eqref{eq_compic2}-\eqref{eq_compic3} and let $T>0.$
A strong solution to \BFG\ on $(0,T),$ associated with  the initial data $(h^0,\dot{h}^0,u^0),$   is a quadruplet $(h,u,p,c)$ satisfying:
\begin{itemize}
\item $h,$ $u,$ $p$ and $c$ have the following regularity:
\begin{equation}\label{reg_h2}
h \in H^2(0,T ; L^2_{\sharp}(0, L)) \cap L^2(0,T; H^4_{\sharp}(0, L)),\,h^{-1} \in L^{\infty}(0,T) \times (0,L)),
\end{equation}
\begin{equation}\label{reg_up}
u \in H^1_{\sharp}(\mathcal Q_T)\,,  \quad \nabla^2 u \in L^2_{\sharp}(\mathcal Q_T)\,, \quad
c \in L^2(0,T)\,, \quad \nabla p \in L^2_{\sharp}(\mathcal Q_T)\,,
\end{equation}
\item equations \eqref{eq_NS}-\eqref{eq_incomp} are satisfied a.e. in $\mathcal Q_T,$
\item equations \eqref{decomp:pression}-\eqref{eq_conditionp0}-\eqref{const:pression} are satisfied a.e. in $(0, T),$
\item equations \eqref{eq_elasticite}-\eqref{eq_bc2}-\eqref{eq_bc1} are satisfied a.e. in $(0, T)\times (0, L),$
\item equations \eqref{eq_ic1}-\eqref{eq_ic2}-\eqref{eq_ic3} are satisfied a.e. in $(0, L)$ and $\mathcal F^0\,.$
\end{itemize}
\end{definition}

\medskip

We emphasize that the pressure $p$ in our definition is completely fixed and that $u$ is a time-space function. Hence, the condition $u \in H^1_{\sharp}(\mathcal Q_T)$ involves both time and space derivatives of $u$, whereas $\nabla p$ involves
space derivatives only.  The construction that we describe above, adapted from \cite{Lequeurre11},  yields the following existence and uniqueness theorem: 
\begin{theorem}  \label{thm_loc}
Let us consider $\alpha >0$,  $\beta\geq 0$  and $\gamma>0$. Assume that the initial data $(h^0,\dot{h}^0,u^0) $ belong to $ H^3_{\sharp}(0, L) \times H^1_{\sharp}(0, L) \times H^1_{\sharp}(\mathcal{F}^0)$  
and  satisfy  the compatibility conditions \eqref{eq_compic1}-\eqref{eq_compic11}-\eqref{eq_compic2}-\eqref{eq_compic3}.
There exists $T_0>0$ such that for any $0<T<T_0,$ there exists a unique strong solution to \BFG\ on $(0,T).$
\end{theorem}
\begin{proof}
In the whole proof initial data $(h^0,\dot{h}^0,u^0)$ are fixed. The only points that we want to make clear here are
\begin{itemize}
\item the link between the regularity \eqref{reg_hatu}-\eqref{reg_hatp}-\eqref{reg_h} of the solution $(\hat{u},\hat{p},h)$ to the nonlinear system in a fixed geometry and
the regularity statements of  {Definition \ref{def_strongsolution}}; 
\item the computation of the Lagrange multiplier $c$ that we introduce here and that does not appear in \cite{Lequeurre11,Lequeurre13}.
\end{itemize} 

\medskip

Let $(\hat{u},\hat{p},h)$  be the solution to  \eqref{eq_NS}-\eqref{eq_incomp}-\eqref{eq_elasticite}-\eqref{eq_phi}-\eqref{eq_bc2}-\eqref{eq_bc1}-\eqref{eq_elasticite'} written in a fixed geometry, that one constructs adapting the
arguments of \cite{Lequeurre11,Lequeurre13} as explained in introduction. The deformation $h(t,\cdot)$ and its inverse $1/h(t,\cdot)$ are then uniformly (w.r.t. $t \in [0,T]$) bounded in  $H^2_{\sharp}(0, L)$ and $L^{\infty}_{\sharp}(0, L)$.  So, we construct $(u,p)$
{\em via} \eqref{def_hat} and apply Proposition \ref{prop_cdv}  to obtain a fluid velocity and a fluid pressure
 that satisfy their contribution to \eqref{reg_up}.  Noting that, in the aforementioned construction, $p$ is defined up to a time-dependent constant and that $\phi(u,p+c,h) = \phi(u,p,h)+c,$ we fix $p$ by requiring further that:   
\begin{equation}
\int_{0}^{L} \phi(u,p,h) {\rm d}x = 0\,, \quad \forall \, t \in (0,T)\,.
\end{equation}
We then write $p= p_0+c$ with $p_0$ satisfying \eqref{eq_conditionp0} and with $c$ being fixed by \eqref{const:pression}. 
Due to the regularity of  $h,$ ${u}$ and $p_0$ we obtain finally that $c$ belongs to $L^2(0, T)$.

\medskip

Conversely,  for any given strong solution $(u,p,h,c)$  of \BFG\,  in the sense of Definition \ref{def_strongsolution}, we  construct $(\hat{u},\hat{p})$ 
by \eqref{def_up} and refer to  Propostion \ref{prop_cdv} again yielding that:
$$
\hat{u} \in H^1(0,T;L^2_{\sharp}(\Omega_1)) \cap L^2(0,T;H^2_{\sharp}(\Omega_1)),
\quad 
\hat{p} \in L^2(0,T ; H^1_{\sharp}(\Omega_1))\,.
$$
We apply then  \cite[Theorem 3.1]{LM72}  and deduce that $\hat{u} \in C([0,T] ; H^1_{\sharp}(\Omega_{1}))$ 
and get that, for $T$ small enough, $(\hat{u},\hat{p},h)$ is the unique solution to 
\eqref{eq_NS}-\eqref{eq_incomp}-\eqref{eq_elasticite}-\eqref{eq_phi}-\eqref{eq_bc2}-\eqref{eq_bc1}-\eqref{eq_elasticite'} written in a fixed geometry, as constructed by adapting the arguments of \cite{Lequeurre11,Lequeurre13}.
\end{proof}
\begin{rem}
\label{rem:reg_sup}
From the regularity we just derived for $\hat u$, we deduce that, for any strong solution $(u, p, h)$ the mapping $t\mapsto \int_{{\mathcal F}(t)} |\nabla u|^2(t)$ belongs to $C^0([0, T])$. 
\end{rem}
Finally, we obtain that \BFG\, is wellposed locally in time.
Following \cite{Lequeurre11}, it appears that we might choose the time $T_0$ in {Theorem \ref{thm_loc}}  to be fixed by   
$$
\|h^0 ;  {H^3_{\sharp}(0,L)} \| + \| \dot{h}^0 \: ; \:  H^1_{\sharp}(0, L) \| + \|{h}^{-1} \: ; \: L^{\infty}_{\sharp}(0, L)\|   + \|{u}^0 ; H^1_{\sharp}(\mathcal F^0)\|
$$
only  (see the computation of $T_0$ at item $(i)$, page 408). Then the following blow-up alternative can be classically stated:
\begin{corollary} \label{cor_blowup}
Let  $\alpha>0$,  $\beta\geq 0$ and  $\gamma>0$ be given. Assume that the initial data $(h^0,\dot{h}^0,u^0)$ belong to $H^3_{\sharp} (0, L)\times H^1_{\sharp}(0, L) \times H^1_{\sharp}(\mathcal{F}^0)$ and  satisfy the  compatibility conditions \eqref{eq_compic1}-\eqref{eq_compic11}-\eqref{eq_compic2}-\eqref{eq_compic3}.  Then \BFG \ completed with initial conditions \eqref{eq_ic1}-\eqref{eq_ic3} has a unique non-extendable strong solution $(T^*,(u,p,h,c)).$ Furthermore, we have the following alternative:\\[-8pt]
\begin{enumerate}
\item[$(i)$] either $T^* = +\infty$\\[-8pt]
\item[$(ii)$]  either $T^* < \infty$ and 
$$
\limsup_{t \to T^*}  \|h(\cdot,t) ;  {H^3_{\sharp}(0,L)} \| +  \| \partial_t{h}(\cdot,t) \: ; \:  H^1_{\sharp}(0, L) \| + \|{h^{-1}(\cdot,t)} \: ; \: L^{\infty}_{\sharp}(0, L)\| +  \|u(\cdot,t); H^1_{\sharp}(\mathcal F(t)) \| = +\infty\,.
$$
\end{enumerate}
\end{corollary}
The aim of {Section \ref{sec_glob}} is to prove that the second alternative $(ii)$ never holds  and consequently that the solution is defined on any finite time inetrval $(0 ,T)$. But before going any further we focus on the elliptic regularity properties of the inhomogenous Stokes system in a subgraph domain.

\subsection{Elliptic estimate}
\label{subsection:elliptique}
In this subsection we derive elliptic estimates for the inhomogeneous Stokes problem in a domain $\Omega_h$ with $h \in H^2_{\sharp}(0, L)$ such that $h^{-1} \in L^{\infty}_{\sharp}(0, L)$.  With this regularity, the domain  is neither $C^{1, 1}$ nor $W^{2,\infty}$ and one cannot apply standard elliptic regularity results. Nevertheless, we take advantage here of the fact that $\Omega_h$ is a subgraph so that the change of variable transforming $\Omega_h$ into a flat domain (namely $\Omega_1$)
can be chosen to be smooth in the transverse variable (see $\chi_h$ below). This remark enables us to extend the classical method with  $h$ belonging merely to $H^2_{\sharp}(0, L).$ Such an estimate is a key argument in the derivation of the regularity estimates for the solution of the non linear system \BFG.

\medskip

For simplicity, we fix $\mu=1$ in this part.
Let us consider source terms  and  a boundary condition 
$$
(f, g) \in L^2_{\sharp}(\Omega_{h})\times H^1_\sharp(\Omega_h)), \quad \dot{\eta}\in H^{\frac 32}_{\sharp}(0, L).
$$
We aim at studying the regularity properties of $L$-periodic (w.r.t. $x$) solutions  to 
\begin{eqnarray}
-{\Delta} u + \nabla p_0 &=& f\,, \quad \text{in }\Omega_{h}, \label{eq_stokes1}\\
{\rm div}\, u &=& g\,, \quad \text{in } \Omega_{h},  \label{eq_stokes2}
\end{eqnarray}
completed with boundary conditions:
\begin{eqnarray}
u(x,h(x)) &=& \dot{\eta}(x) e_2 \,, \phantom{0 \,,} \quad \forall \, x \in (0, L)\,, \label{eq_stokesbc1}\\
u(x,0)&=& 0 \,, \phantom{\dot{h}(x)\,,} \quad \forall \, x \in (0, L)\,. \label{eq_stokesbc2}
\end{eqnarray}
Integrating ${\rm div} \, u = g$ over $\Omega_{h}$ implies that the  boundary velocity $\dot{\eta}$ has to satisfy
\begin{equation} \label{eq_stokescomp}
\int_{0}^{L} \dot{\eta}(x){\rm d}x = \int_{\Omega_h} g(x,y){\rm d}x{\rm d}y\,.
\end{equation}
The left-hand side of  \eqref{eq_stokescomp} does not involve the deformation $h$, because the deformation as well as the boundary velocity are vertical.
In what follows, we restrict to data $g \in L^2_{\sharp,0}(\Omega_h)$ and $\dot{\eta} \in L^2_{\sharp,0}(0,L)$ for which \eqref{eq_stokescomp} is clearly
satisfied.
\begin{rem}
Note that, for this inhomogeneous Stokes problem, with Dirichlet boundary conditions,  the pressure $p_0$ is defined up to a constant. Consequently, we enforce uniqueness of the pressure by imposing:
 $$
 \int_{\Omega_h} p_0(x, y) {\rm d}x{\rm d}y =0.
 $$
\end{rem}
The main result of this section writes
\begin{lemma} \label{lem_ellest}
For any $h \in H^2_{\sharp}(0, L)$ such that $h^{-1} \in L^{\infty}_{\sharp}(0, L),$  source terms and boundary condition 
$$
(f, g) \in L^2_{\sharp}(\Omega_{h})\times (H^1_\sharp(\Omega_h)\cap L^2_{ \sharp, 0}(\Omega_h)), \quad \dot{\eta}\in  H^{\frac 32}_{\sharp}(0,L) \cap L^2_{\sharp,0}(0,L),
$$ 
there exists a unique solution $(u,p_0) \in H^2_{\sharp}(\Omega_{h}) \times (H^1_{\sharp}(\Omega_h)\cap L^2_{ \sharp, 0}(\Omega_h))$ to the Stokes system \eqref{eq_stokes1}-\eqref{eq_stokes2}-\eqref{eq_stokesbc1}-\eqref{eq_stokesbc2}. 
 Moreover, there exists a non-decreasing function $K^s : [0,\infty) \rightarrow (0,\infty)$ such that, if we assume $\|h ; H^2_{\sharp}(0, L)\| + \|h^{-1} ; L^{\infty}_{\sharp}(0, L)\| \leq R_0$ then, this solution satisfies:
\begin{equation} \label{eq_ellest}
\|u ; H^2_{\sharp}(\Omega_h)\| + \| p_0 ; H^1_{\sharp}(\Omega_h)\| \leq K^s(R_0)\left( \|f ; L^2_{\sharp}(\Omega_h)\| +\|g ; H^1_{\sharp}(\Omega_h)\| + \|\dot{\eta} ; H^{\frac 32}_{\sharp}(0, L)\| \right)\,.
\end{equation}

\end{lemma}

The remainder of this section is devoted to the proof of  Lemma \ref{lem_ellest}. This proof is an adaptation, 
to our periodic framework, of the computations on the Stokes problem that can be found in \cite{Grandmontpp}, which, itself, uses ideas of  \cite{tartar}. Compared to \cite{Grandmontpp}, we also carefully analyze the dependance of the constant 
$K$ on $h$ in inequality  \eqref{eq_ellest}.  To obtain the expected dependency, we assume throughout this section that:
$$
\|h ; H^2_{\sharp}(0, L)\| + \|h^{-1} ; L^{\infty}_{\sharp}(0, L)\| \leq R_0
$$
and show that $K$ depends only on $R_0.$

\medskip

\paragraph{\em First step:  Rewriting of the Stokes system in a given geometry}
As in \cite{Grandmontpp} we compute regularity estimates for solutions to \eqref{eq_stokes1}--\eqref{eq_stokesbc2} by studying the Stokes system transported in a geometry which does not depend on the deformation $h$. Namely, we
derive regularity estimates on $(\hat{u},\hat{p})$  defined by \eqref{def_up}. Indeed, thanks to Proposition \ref{prop_cdv}, we remark that  $(u,p) \in H^2_{\sharp}(\Omega_{h}) \times H^1_{\sharp}(\Omega_h)$ is a  solution 
to  \eqref{eq_stokes1}--\eqref{eq_stokesbc2} if and only if $(\hat{u},\hat{p}) \in H^2_{\sharp}(\Omega_{1}) \times H^1_{\sharp}(\Omega_1)$ is a solution 
to the following Stokes-like system 
\begin{eqnarray}
-{\rm div} [ (A_h \nabla) \hat{u} ] +(B_{h} \nabla) \hat{p}_0 &=& \tilde{f}\,, \quad\mbox{ in }\Omega_{1}\,, \label{eq_hstokes1}\\
{\rm{div }}(B_{h}^{\top}  \hat{u})&=& \tilde{g}\,, \quad \mbox{ in }\Omega_{1}\,,  \label{eq_hstokes2}
\end{eqnarray}
completed with boundary conditions:
\begin{eqnarray}
\hat{u}(x,1) &=& \dot{\eta}(x) e_2\,, \phantom{0 \,,} \quad \forall \, x \in (0, L)\,, \label{eq_hstokesbc1}\\
\hat{u}(x,0)&=& 0 \,, \phantom{\dot{\eta}(x)\,,} \quad \forall \, x \in (0, L)\,, \label{eq_hstokesbc2}
\end{eqnarray}
where $(A_h,B_h)$ and $(\tilde{f},\tilde{g})$ are explicit. 
Indeed, by introducing the mapping $\chi_h(x, z)=(x, h(x)z)$, for $(x, z)\in \Omega_1$, we obtain:
\begin{equation}
\label{matrice:h}
B_{h} :=\textrm{cof }\nabla \chi_h=
\begin{pmatrix}
h & -h' z \\
0 & 1
\end{pmatrix},
\quad 
A_h := \frac{1}{h} (\textrm{cof }\nabla \chi_h)^T\  \textrm{cof }\nabla \chi_h=
\begin{pmatrix}
h & -h'z \\
-h'z & \frac{1}{h} + \frac{(h'z)^2}{h}
\end{pmatrix},
\end{equation}
and the transported source terms: 
$$
\tilde{f} := h
\hat{f},\quad
\tilde{g} := h
\hat{g}.
$$

We note that ${\tilde f}\in L^2_\sharp (\Omega_1)$, $\tilde{g}\in L^{2}_{\sharp, 0}(\Omega_1)\cap  H^1_\sharp (\Omega_1)$. Thus, thanks to Proposition \ref{prop_cdv}, to prove Lemma \ref{lem_ellest} it is sufficient to derive similar estimates but on the transported unknowns $(\hat{u},\hat{p})$ solution of \eqref{eq_hstokes1}--\eqref{eq_hstokesbc2}. Namely, we obtain that there exists a unique $({\hat u}, {\hat p}_0)\in H^2_\sharp(\Omega_1)\times (H_\sharp^1(\Omega)\cap L^2_{\sharp, 0}(\Omega_1))$ solution of \eqref{eq_hstokes1}--\eqref{eq_hstokesbc2} that satisfies
\begin{equation}
\label{eq_ellest2}
\|{\hat u }; H^2_{\sharp}(\Omega_1)\| + \| {\hat p }_0; H^1_{\sharp}(\Omega_1)\| \leq { K }\left( \|{\tilde f }; L^2_{\sharp}(\Omega_1)\| +\|{\tilde g} ; H^1_{\sharp}(\Omega_1)\| + \|\dot{\eta} ; H^{\frac 32}_{\sharp}(0, L)\| \right)\,,
\end{equation}
where the constant $K$ depends only on  $R_0$.  Since the matrices $B_h$ and $A_h$ are in $H^1_\sharp((0,L);H^s(0, 1)),$ for any $s\geq 0$,
we have that $A_h$ and $B_h$ belong to a multiplier space of $H^1(\Omega_1)$. We refer the reader to \cite[Lemma 6]{Grandmontpp} for more details.
In particular, we obtain that for any $v\in H_\sharp^2(\Omega_1)$,  there holds ${\rm div} [ (A_h \nabla) v]\in L_\sharp^2(\Omega_1)$ and, for any $q\in H_\sharp^1(\Omega_1)$, there holds
$(B_h\nabla) q\in L_\sharp^2(\Omega_1)$.  Thanks to Piola identity, we also have:
\begin{equation} \label{eq_piola}
 {\rm div}(B_h^{\top} v )=B_{h}^{\top}  : \nabla  v ,
\end{equation}
so that,  for any $v\in H_\sharp^2(\Omega_1)$, there holds $ {\rm div}(B_h^{\top} v  )\in H_\sharp^1(\Omega_1)$.   Consequently the assumptions on the deformation $h$ are compatible with the expected regularity on $(\hat u, \hat p)$.

\begin{rem} Let us mention that we  get estimates for a pressure $\hat{p}_0$ such that 
$$
\int_{\Omega_1} \hat{p}_0(x,z){\rm d}x {\rm d}z = 0\,.
$$
Through the change of variables \eqref{def_hat}, this implies that the pressure $q$,  defined by $q(x, y)={\hat p}_0(x, y/h(x))$ and on which we deduce an estimate, verifies the following constraint
$$
\int_{\Omega_{h}} \dfrac{q(x,y)}{h(x)}{\rm d}x {\rm d}y = 0\,. 
$$
So, the pressure we  compute with this method does not match the one mentioned in  Lemma \ref{lem_ellest}.   Nevertheless,  
the effective pressure $p_0$ mentioned in Lemma \ref{lem_ellest}  reads $P_{h}q$ where $P_h$ stands for the $L^2$-orthogonal
projector on   $\L^2_{\sharp, 0}(\Omega_h)$:  
$$
p_0=P_h q := q - \dfrac{1}{|\Omega_{h}|} \int_{\Omega_h} q(x,y){\rm d}x{\rm d}y
$$
satisfying $\|p_0 ; H^1_{\sharp}(\Omega_h)\| \leq \|q ; H^1_{\sharp}(\Omega_h)\|.$  Hence, we prove \eqref{eq_ellest}
with $p_0$ replaced by $q$. 
\end{rem}

Let us now  study precisely the existence and uniqueness of $({\hat u}, {\hat p})$ in $H_\sharp^1(\Omega_1)\times L^2_{\sharp, 0}(\Omega_1)$ and derive elliptic estimates in $H_\sharp^2(\Omega_1)\times H_\sharp^1(\Omega_1)$.

\medskip

\paragraph{\em Second step: Lifting of the Dirichlet boundary conditions}
As $\dot{\eta} \in H^{\frac 32}_{\sharp}(0,L)$ there exists $u_{\dot \eta} \in H^2_\sharp (\Omega_1)$ such that ${u_{\dot \eta}}_{|_{z=1}}=\dot\eta e_2$, ${u_{\dot \eta}}_{|_{z=0}}=0$ and 
\begin{equation}
\label{est:doteta}
\|u_{\dot \eta};  H^2_\sharp (\Omega_1) \|\leq C\| \dot\eta ; H^{\frac 32}_\sharp(0, L)\|.
\end{equation}
with a constant $C$ depending only on the fixed geometry.
 We  set
$\bar u= \hat u -{u_{\dot \eta}}$. This new velocity satisfies:
\begin{eqnarray}
-{\rm div} [ (A_h \nabla) \bar{u} ] +(B_{h} \nabla) \hat{p}_0 &=&\bar f\,, \quad \mbox{ in } \Omega_{1}\,,\label{Stokes:newf}\\
{\rm{div }}(B_{h}^{\top}  \bar{u})&=& \bar g\,, \quad  \mbox{ in }\Omega_{1}\,,  \label{Stokes:newg}
\end{eqnarray}
with
$$
\bar f :=  \tilde{f}+{\rm div} [ (A_h \nabla) {u}_{\dot\eta} ]\,,  \qquad \bar g:=\tilde{g}- B_h^{\top} : \nabla {u}_{\dot\eta} \,,
$$
completed with boundary conditions:
\begin{eqnarray}
\bar{u}(x,1) &=& 0\,, \phantom{0 \,,} \quad \forall \, x \in (0, L)\, \label{Stokes:DH1}\,,\\
\bar{u}(x,0)&=& 0 \,, \phantom{\dot{\eta}(x)\,,} \quad \forall \, x \in (0, L)\,.\label{Stokes:DH2}
\end{eqnarray}
As underlined previously, thanks to the regularity of $A_h$ and $B_h$, the new sources terms $(\bar f, \bar g)$ belong to 
$L^2_\sharp(\Omega_1)\times H^1_\sharp(\Omega_1)$ and satisfy the following estimates:
\begin{equation}
\label{est:termesource}
\|{\bar f}; L^2_\sharp(\Omega_1)\|+ \|{\bar g}; H^1_\sharp(\Omega_1)\|\leq K(\|{\tilde f}; L^2_\sharp(\Omega_1)\|+ \|{\tilde g};H^1_\sharp(\Omega_1)\| + \| \dot\eta ; H^{\frac 32}_\sharp(0, L)\|),
\end{equation}
where $K$ depends only on $R_0$. Moreover the average of $\bar g$ on $\Omega_1$ is still equal to zero, since
$$
\int_{\Omega_1} B_h^{\top} : \nabla {u}_{\dot\eta} =\int_{\Omega_1} \mbox{div } (B_h^{\top} {u}_{\dot\eta} )= \int_0^L \dot \eta=0.
$$
Recalling that $u_{\dot \eta}$ satisfies \eqref{est:doteta}, we obtain now that the proof of Lemma \ref{lem_ellest} reduces to the study of the case $\dot{\eta} = 0$
({\em i.e.} solving system \eqref{Stokes:newf}-\eqref{Stokes:newg}-\eqref{Stokes:DH1}-\eqref{Stokes:DH2}).
 \medskip

\paragraph{\em Third step: $H^1\times L^2$ estimates} We first define $H^{-1}_\sharp(\Omega_1)$ as the dual space of the subset of  $H^{1}_\sharp(\Omega_1)$ of functions  with zero trace on $(0, L)\times \{0\}$ and $(0, L)\times \{1\}$.
The aim of this step is to prove that, for any $(\bar f, \bar g)\in H^{-1}_\sharp(\Omega_1)\times L^2_{\sharp,0}(\Omega_1)$,  there exists a unique $(\bar u, \hat p_0)\in H^1_\sharp(\Omega_1)\times L^2_{\sharp,0}(\Omega_1)$ solution of  \eqref{Stokes:newf}-\eqref{Stokes:DH2} and satisfying
\begin{equation}
\label{eq_stokes:weak}
\|{\bar u }; H^1_{\sharp}(\Omega_1)\| + \| {\hat p }_0; L^2_{\sharp}(\Omega_1)\| \leq { K }\left( \|{\bar f }; H^{-{1}}_{\sharp}(\Omega_1)\| +\|{\bar g} ; L^2_{\sharp}(\Omega_1)\| \right)\,,
\end{equation}
where $K$ depends only on $R_0$.
Since the arguments are quite standard (see, for instance, \cite{G98}, \cite{Grandmontpp} in similar contexts), we only sketch the main points of the proof.
First, we notice that:
\begin{itemize}
\item $A_h\in L^\infty(\Omega_1)$ and there exists  two non negative constants
 $\alpha_1$ and $\alpha_2$  controlled by above and from below by a function of $R_0$ for which 
 $$
\alpha_1  {\rm I} \leq A_{h} (x,z)\leq \alpha_2{\rm I}\,, \quad \forall \, (x,z)  \in \Omega_{1}\,,
 $$
 in the sense of symmetric matrices;
\item $B_h$ is invertible and $B_h^{-1}$ belongs to $H^1_\sharp((0,L);H^s(0, 1)) ,$ for any $s\geq 0$, with norms  dominated by a function of  $R_0$ only.
\end{itemize}
With the second point at-hand, we build a lifting operator for the divergence. Namely, for any $\chi\in L^2_{\sharp, 0}(\Omega_1)$, there exists a function $w\in H^1_{\sharp}(\Omega_1)$, with $w_{|_{z=1}}=w_{|_{z=0}}=0$, such that 
\begin{equation}
\label{est:w}
{\div } (B_h^{\top} w)=\chi, \quad \|w: H^1_\sharp(\Omega_1)\|\leq K \|\chi; L^2_\sharp(\Omega_1)\|,
\end{equation}
where $K$ depends only on  $R_0$.
 Indeed, as $\chi$ has zero average on $\Omega_1,$ there exists $v\in H^1_{\sharp}(\Omega_1)$, with $v_{|_{z=1}}=v_{|_{z=0}}=0$, such that 
$$
{\div } (v)=\chi, \quad \|v: H^1_\sharp(\Omega_1)\|\leq C\|\chi; L^2_\sharp(\Omega_1)\|.
$$
See, for instance, \cite[Lemma III.3.1]{Galdi}. We set then $w= B_h^{-\top}v.$ As $B_h^{-\top}$ is a multiplier of $H^{1}$ with norm  bounded by a function of  $R_0$ we obtain \eqref{est:w}.

\medskip

Then, to solve \eqref{Stokes:newf}-\eqref{Stokes:DH2}  we first lift the divergence source term $\bar{g}$ by applying the previous construction. We then
solve the Stokes-like system  \eqref{Stokes:newf}-\eqref{Stokes:newg} by reproducing the classical arguments for the Stokes system. 
As  $A_h$ satisfies the first point, we first construct a weak solution $\bar u \in H^1_\sharp(\Omega_1)$ depending continuously on $(\overline{f},\overline{g})$. Then, 
as $B_h^{-1}$ satisfies the second point, we obtain also the pressure  $\hat p_0 \in L^2_{\sharp, 0}(\Omega_1)$ which completes  \eqref{eq_stokes:weak}.
\medskip

\paragraph{\em Fourth step: proof of Lemma \ref{lem_ellest}, $H^2/H^1$-regularity}
To complete the proof of Lemma \ref{lem_ellest}, it remains to obtain an estimate on the second order derivatives of $\bar{u}$ and the first order derivatives of $\hat p_0.$
We obtain that 
\begin{equation}
\label{eq_stokes:strong}
\|{\bar u }; H^2_{\sharp}(\Omega_1)\| + \| {\hat p }_0; H^1_{\sharp}(\Omega_1)\| \leq { K }\left( \|{\bar f };  L^2_{\sharp}(\Omega_1)\| +\|{\bar g} ; H^1_{\sharp}(\Omega_1)\| \right)\,,
\end{equation}
where $K$ depends only on  $R_0$.
We follow the method introduced in \cite{tartar} and already applied in \cite{Grandmontpp} in our subgraph framework. 
Thanks to a classical regularization argument, we assume in what follows that  $h \in C^\infty_{\sharp}(0, L)$.  In this case, classical elliptic estimates ensure that $({\bar u },   \hat p_0)  \in H^2_{\sharp}(\Omega_1)\times H^1_{\sharp}(\Omega_1)$. Nervertheless the standard elliptic estimates involve norms of the deformation in $W^{2, \infty}(0, L)$ (see for instance \cite{Bello}).  Consequenlty, we aim to show that the constant only involves $R_0$. 
\medskip

First we obtain estimate on $\bar{u}_x :=  \partial_x \bar u$ and $\hat{p}_x := \partial_x \hat p_0$.  For this purpose, we differentiate 
the equations \eqref{Stokes:newf}, \eqref{Stokes:newg} satisfied by $(\bar u, \hat p_0)$ w.r.t. $x.$ We obtain that $(\bar{u}_x,\hat{p}_x) \in H^1_{\sharp}(\Omega_1) \times L^2_{\sharp, 0}(\Omega_1)$ is the solution of
\begin{eqnarray*}
-{\rm div} [(A_h \nabla) \bar{u}_x ] +(B_{h} \nabla) \hat{p}_x &=& \bar f_x\,, \quad  \text{on $\Omega_{1}$}\,, \\
\textrm{div }(B_{h}^{\top}  \bar{u}_x )&=& \partial_x\bar g- \partial_x B_h^{\top} : \nabla \bar{u} \,, \quad  \text{on $\Omega_{1}$}\,,  
\end{eqnarray*}
where 
$\bar f_x = \partial_x \bar{f} + {\rm div} [(\partial_x A_h \nabla) \bar{u}] - (\partial_x B_{h} \nabla) \hat{p}_0\,,$ completed with periodic boundary
conditions on lateral boundaries of $\Omega_1$ and homogeneous boundary conditions on $y=1$ and $y=0$ (we recall that we consider the case 
$\dot{\eta}= 0$). 

\medskip

We note that 
$$
 \partial_x B_h^{\top} : \nabla \bar{u} = {\rm div} ( \partial_x B_h^{\top}\bar{u}),
$$
which implies
$$
\int_{\Omega _1} \partial_x B_h^{\top} : \nabla \bar{u}  = 0\,.
$$
Consequently, taking into account that $\bar{g}$ is $L$-periodic  w.r.t. $x$, we obtain that $\bar g_x=\partial_x\bar g- \partial_x B_h^{\top} : \nabla \bar{u}$ has a zero average on $\Omega_1$. Due to the regularity of $(\bar u,\hat p_0)$, the right hand side $(\bar f_x, \bar g_x)$ belongs to $H^{-1}_\sharp(\Omega_1)\times L^2_\sharp(\Omega_1),$ but we need sharp estimate to show our main result.  
As we stated  previously  $(\partial_x A_h,\partial_x B_h) \in L^2((0,L) ,; H^s(0,1)),$ for arbitrary $s \geq 0,$ (with norms bounded by a function of $R_0$) and  
$H^1_{\sharp}((0,L) \times (0,1)) \subset L^{\infty}_{\sharp}((0,L) ; L^2(0,1)).$ Hence, in the spirit of  \cite[Lemma 6]{Grandmontpp}, we obtain
\begin{equation}
\label{second:membre1}
\|\textrm{div }((\partial_x A_h \nabla ) \bar u); H^{-1}_\sharp(\Omega_1)\|\leq \|(\partial_x A_h \nabla ) \bar u; L^2_\sharp(\Omega_1)\|\leq K \|\bar u; H^1_\sharp(\Omega_1)\|^{1/2}\|\bar u_x; H^1_\sharp(\Omega_1)\|^{1/2},
\end{equation}
and 
\begin{equation}
\label{second:membre2}
\|\partial_x B_h: \nabla  \bar u; L^2_\sharp(\Omega_1)\|\leq K \|\bar u; H^1_\sharp(\Omega_1)\|^{1/2}\|\bar u_x; H^1_\sharp(\Omega_1)\|^{1/2},
\end{equation}
where $K$ depends on $R_0.$
Next we have to estimate $(\partial_x B_h\nabla) \hat p_0$ in $H^{-1}_\sharp(\Omega_1)$.
Thanks to the Piola identity and the fact that $B_h$ is the cofactor matrix of the gradient of $\chi_h$, we obtain, for any $w\in H^1_\sharp(\Omega_1)$ such that $w_{|_{z=0}}=w_{|_{z=1}}=0$
$$
\int_{\Omega_1} (\partial_x B_h\nabla) \hat p_0 w = - \int_{\Omega_1}  \hat p_0 \partial_xB_h^{\top}:\nabla w.
$$
Consequently, as in the computations of the latter bounds, we obtain:
\begin{equation}
\label{second:membre3}
\|(\partial_x B_h\nabla) \hat p_0; H^{-1}_\sharp(\Omega_1)\|\leq K \|\hat p_0; L^2_\sharp(\Omega_1)\|^{1/2}\|\hat p_x; L^2_\sharp(\Omega_1)\|^{1/2}.
\end{equation}
We can now apply the result obtained at the previous  step to $(\bar{u}_x,\hat{p}_x) $. Combining with \eqref{second:membre1}, \eqref{second:membre2}, \eqref{second:membre3}, this leads to
\begin{eqnarray*}
\|\bar u_x; H^1_\sharp(\Omega_1)\|+\|\bar p_x; L^2_\sharp(\Omega_1)\|\leq K \left( \|\bar f; L^2_\sharp(\Omega_1)\|+\|\bar g; H^1_\sharp(\Omega_1)\|+
 \|\bar u; H^1_\sharp(\Omega_1)\|^{1/2}\|\bar u_x; H^1_\sharp(\Omega_1)\|^{1/2}  \right.\nonumber\\\left.+\|\hat p_0; L^2_\sharp(\Omega_1)\|^{1/2}\|\hat p_x; L^2_\sharp(\Omega_1)\|^{1/2}|\right).
\end{eqnarray*}
and finally: 
\begin{eqnarray}
\|\bar u_x; H^1_\sharp(\Omega_1)\|+\|\bar p_x; L^2_\sharp(\Omega_1)\|\leq K \left( \|\bar f; L^2_\sharp(\Omega_1)\|+\|\bar g; H^1_\sharp(\Omega_1)\|\right).\label{est:dx}
\end{eqnarray}

\medskip

 To obtain a similar estimate on the full second order gradient of $\bar u$ (resp. on the full gradient of $\hat p_0$), we have to bound $\partial_{zz} \bar{u}$ (resp. $\partial_z \hat p_0$).
 To this end, we note that, differentiating \eqref{Stokes:newg} w.r.t $z$ and applying \eqref{est:dx}, we have 
 $$
 \| - z h'\partial_{zz} \bar{u}_1 +  \partial_{zz} \bar{u}_2 ; L^2_{\sharp}(\Omega_1)\| \leq K \left( \|\bar{f};L^2_{\sharp}(\Omega_1)\| +\|\bar{g};L^2_{\sharp}(\Omega_1)\|  \right) .
 $$
While, combining the first equation of \eqref{Stokes:newf} with the second equation of \eqref{Stokes:newf}
multiplied by $zh'$, in order to eliminate the pressure, leads to
$$
 \| {z h'} \partial_{zz} \bar{u}_2 + \partial_{zz} \bar{u}_1 ; L^2_{\sharp}(\Omega_1)\| \leq K ( \|\bar f; L^2_\sharp(\Omega_1)\|+\|\bar g; H^1_\sharp(\Omega_1)\|).
$$ 
By simple algebraic combinations,   since $1+z^2 (h')^2 >0$,  we obtain 
$$
\|\nabla^2 \bar{u} ; L^2_{\sharp}(\Omega_1)\|  \leq K \left( \|\tilde{f};L^2_{\sharp}(\Omega_1)\| + \|\bar{g};L^2_{\sharp}(\Omega_1)\| \right).
$$
A similar inequality holds for $\|\partial_z \tilde{p} ; L^2_{\sharp}(\Omega_1)\|$, using once again the first equation of \eqref{Stokes:newf}.
 Combining these inequalities,  we  finally  obtain the desired bound
$$
\| \bar{u} ; H^2_{\sharp}(\Omega_1)\|  + \|\hat {p}_0 ; H^1_{\sharp}(\Omega_1)\|  \leq K \left( \|\bar{f};L^2_{\sharp}(\Omega_1)\| +  \|\bar{g};L^2_{\sharp}(\Omega_1)\| \right).  
$$ 
This ends the proof of Lemma \ref{lem_ellest}.

\medskip

For the study of the whole coupled system, we need an estimate on the surface load applied by the fluid on the structure.
If we compute $\phi(u,p,h)$ through the change of variable \eqref{def_up} ({\em i.e.} w.r.t. $\hat{u},\hat{p}$ and $h,$) and we use, for instance, the multiplier Lemma \cite[Proposition B.1]{GrubbSolonnikov} or Proposition \ref{prop_cdv}, we can also obtain the following corollary, stated without proof:
\begin{corollary} \label{cor_phi}
Let $h \in H^2_{\sharp}(0, L)$ such that $h^{-1} \in L^{\infty}_{\sharp}(0, L)$ be given  there exists a non decreasing function $K^b: [0,\infty) \rightarrow (0,\infty)$ such that, if $\|h ; H^2_{\sharp}(0, L)\| + \|h^{-1} ; L^{\infty}_{\sharp}(0, L)\| \leq R_0$ the following propositions hold true.  

Given source terms $(f , g)\in L^2_{\sharp}(\Omega_{h})\times (H^1_\sharp(\Omega_1)\cap L^2_{{\sharp}, 0}(\Omega_{h})) $ and a boundary condition $\dot{\eta}\in H^{\frac 32}_{\sharp}(0,L) \cap L^2_{\sharp,0}(0, L)$ satisfying \eqref{eq_stokescomp},  the unique pair $(u,p_0) \in H^2_{\sharp}(\Omega_{h}) \times (H^1_{\sharp}(\Omega_h) \cap L^2_{{\sharp}, 0}(\Omega_{h}))$ solution to \eqref{eq_stokes1}-\eqref{eq_stokes2}-\eqref{eq_stokesbc1}-\eqref{eq_stokesbc2}  satisfies:
\begin{equation} \label{eq_phiest}
\|\phi(u,p_0,h) ; H^{\frac 12}_{\sharp}(0, L)\| \leq K^b(R_0) \left( \|f ; L^2_{\sharp}(\Omega_h)\| +\|g ; H^1_{\sharp}(\Omega_h)\| +\|\dot{\eta} ; H^{\frac 32}_{\sharp}(0, L)\| \right)\,.
\end{equation}
The constant $c$ defined by $c=\frac{1}{L}\int_0^L \phi(u, p_0, h) {\rm d}x$ satisfies
\begin{equation}
\label{est:const:pression}
|c|\leq K^b(R_0)\left( \|f ; L^2_{\sharp}(\Omega_h)\| +\|g ; H^1_{\sharp}(\Omega_h)\| + \|\dot{\eta} ; H^{\frac 32}_{\sharp}(0, L)\| \right)\,.
\end{equation}
\end{corollary}

\section{Proof of {Theorem \ref{thm_main}}} \label{sec_glob}
Let  $(h^0,\dot{h}^0,u^0) \in H^3_{\sharp} (0, L)\times H^1_{\sharp} (0, L) \times H^1_{\sharp}(\mathcal F^0)$ be given and 
satisfy the compatibility conditions \eqref{eq_compic1}-\eqref{eq_compic3}. We consider $(u,p,h,c)$, the associated non-extendable strong solution 
to \BFG\ completed with the initial conditions \eqref{eq_ic1}-\eqref{eq_ic3} (in the sense of Definition \ref{def_strongsolution}).  This solution is defined on some time-interval $[0,T^*),$ where $T^*>0$. We compute estimates satisfied by this solution on $[0,T]$ for arbitrary $T<T^*.$ 

\medskip

The proof of Theorem \ref{thm_main} is divided into three parts, each of them corresponding to the derivation of one estimate similar to \eqref{reg_rho1}, \eqref{reg_rho2} and \eqref{reg_rho3} respectively.  First, we recall the energy estimate satisfied by the solution, then we prove a distance estimate which ensures that the beam does not touch the bottom of the fluid cavity on the time interval $(0, T)$ and finally we derive a regularity estimate which garantees that the strong solution can be extended on any given time  interval, leading to our global-in-time existence theorem.

\subsection{Energy estimate}
We first recall the classical  estimate associated with the dissipative equations  that we consider.
We introduce $\mathcal{E}_c$  and $\mathcal H$ respectively the  total energy of the coupled system and the dissipated energy: 
\begin{eqnarray*}
\mathcal{E}_c (t) &:=& \dfrac{1}{2} \left[ \int_{0}^{L} \left( \rho_s|\partial_{t} h(x,t)|^2 + \alpha |\partial_{xx} h(x,t) |^2 + \beta |\partial_{x} h(x,t)|^2 \right){\rm d}x + \int_{\mathcal F(t)} \rho_f|u(x,y,t)|^2 {\rm d}x {\rm d}y  \right]\,,  \\
\mathcal H(t) &:=&\gamma \int_{0}^{L} |\partial_{tx} h(x,t)|^2 {\rm d}x + \mu \int_{\mathcal F(t)} |\nabla u(x,y,t)|^2 {\rm d}x {\rm d}y\,.
\end{eqnarray*}
We have then that:
\begin{proposition} \label{prop_estkin}
The following energy balance holds true
\begin{equation} \label{decay_ec}
\mathcal{E}_c(t) + \int_{0}^t \mathcal H(s){\rm d}s = \mathcal {E}_c(0)\,, \quad \forall \, t \in [0,T]\,.
\end{equation}
\end{proposition}

\begin{proof}
Multiplying  first \eqref{eq_NS} by $u$ and integrating by parts leads to
\begin{eqnarray}
\int_0^t\int_{{\mathcal F}(s)} (\partial_t u + u \cdot \nabla u )\cdot u &=& \int_0^t\int_{{\mathcal F}(s)} {{\rm div}\sigma(u,p)} \cdot u \notag \\
											&=& \int_0^t \int_{\partial \mathcal F (s)} \sigma(u,p)n \cdot u - 2\mu \int_0^t\int_{{\mathcal F}(s)} |D(u)|^2 \notag\\
											&=&  -\int_0^t \int_{0}^{L} \phi(u,p,h) \partial_t h - 2\mu \int_0^t\int_{{\mathcal F}(s)} |D(u)|^2\,. \label{eq_flux}
\end{eqnarray}
In the last equality we have used the coupling conditions at the interface between the fluid and the structure.
Moreover  thanks to  the only vertical motion of the beam together with the divergence free constrain (see \cite[Lemma 6]{CDEG05} in the 3D case),
there holds:
\begin{eqnarray}
2 \int_{\mathcal F(s)} |D(u)|^2 &=& \int_{\mathcal F (s)} |\nabla u|^2\,. \label{eq_Korn}
\end{eqnarray}
Furthermore, since the boundaries of $\mathcal{F}(t)$ move with the velocity-field $u,$ we have
\begin{equation} \label{eq_ke1}
\int_0^t\int_{{\mathcal F}(s) }(\partial_t u + u \cdot \nabla u )\cdot u =  \dfrac{1}{2} \left[\int_{\mathcal F (s)}  |u|^2 \right]_{s=0}^{s=t}\,,
\end{equation}
Consequenlty 
\begin{eqnarray}
 \label{eq_flux2}
\dfrac{1}{2} \int_{\mathcal F (t)}  |u|^2+\mu \int_0^t\int_{\mathcal F (s)} |\nabla u|^2=\dfrac{1}{2} \int_{\mathcal F^0}  |u^0|^2 -\int_0^t \int_{0}^{L} \phi(u,p,h) \partial_t h.
\end{eqnarray}
If we now multiply the beam equation \eqref{eq_elasticite} by $\partial_t h$ we obtain, after time and space integration by parts,
\begin{multline} \label{eq_ke2}
 \dfrac{1}{2} \left[\int_{0}^{L} \left( \rho_s|\partial_{t} h|^2 + \alpha |\partial_{xx} h|^2 + \beta |\partial_{x} h|^2 \right)\right]_{s=0}^{s=t} 
+ \gamma \int_{0}^t \int_{0}^{L} |\partial_{tx} \eta|^2= \int_0^t  \int_{0}^{L} \phi(u,p,h) \partial_t h \,.
\end{multline}
By summing   \eqref{eq_flux2} and \eqref{eq_ke2}, we obtain the expected result.
\end{proof}

\subsection{Distance estimate}  \label{sec_distance} The aim of this section is to prove the following proposition
\begin{proposition} \label{prop_distest}
There exists a constant $C_{0}$ depending only on initial data  for which:
$$
\sup_{t  \in (0,T)} \Big(\gamma\| h(t,\cdot) \, ; \, H^2_{\sharp}(0,L)\|^2 + \|h^{-1}(t,\cdot) \, ; \, L^1_{\sharp}(0,L)\| \Big) +\alpha  \int_0^T\| h(t,\cdot) \, ; \, H^3_{\sharp}(0,L) \|^2 {\rm d}t \leq C_{0}(1+T)\,.
$$
\end{proposition}
The remainder of this paragraph is devoted to the proof of this result. Let us consider $0\leq t\leq T$.  In all what follows $C_0$ denotes a constant depending only on the initial data but which may change between lines.
Let $w = \nabla^{\bot} \psi = (-\partial_y \psi, \partial_x \psi)$ where:
$$
\psi(x,y,t) = \partial_{x}h(x,t) \chi_0\left(\dfrac{y}{h(x,t)}\right)\,, \quad \forall \, (x,y,t) \in \mathcal Q_T\,,
$$
with: 
$$
\chi_0(z) = z^2(3-2z) \,, \quad   \forall \, z \in (0,1)\,.
$$
Combining the regularity of $h$ (which implies that $h \in C([0,T];H^3_{\sharp}(0,L)) \cap H^1(0,T;H^2_{\sharp}(0,L))$ thanks to 
\cite[Theorem 3.1]{LM72} see \eqref{reg_h2}) 
with $\chi_0 \in C^{\infty}([0,1])$ we obtain that  $w\in H^1(\mathcal Q_T)$ and $\nabla^2 w\in L^2(\mathcal Q_T).$ This regularity is enough to justify all computations below as $(u,p,h,c)$ is a strong solution.  Moreover,   $w$ is divergence free by construction and, since $\chi_0(1)=0$ and $\chi_0'(1) = \chi_0(0)=\chi_0'(0)=0$, there holds:
$$
\begin{array}{rcll}
w(x,h(t,x),t) &=&   \partial_{xx} h(t,x) e_2, & \text{ for all $x \in (0,L)$ and $t  \in (0,T),$ }     \\[4pt]
w(x,0,t) &=& 0\,,  &\text{ for all $x \in (0,L)$ and $t \in (0,T)$}\,.
\end{array}
$$

Consequently, we multiply \eqref{eq_NS} by $w$ and \eqref{eq_elasticite} by {$\partial_{xx} h $} and integrate on $\mathcal Q_t$ for arbitrary $t<T$. 
We get after integration by parts (note that the terms involving the fluid/beam interactions cancel out since the structure test function is the trace of the fluid test function on the interface):
\begin{eqnarray} 
 -\int_{\mathcal Q_t} \rho_f (\partial_{t} u + u \cdot \nabla u ) \cdot w - 2 \mu \int_{\mathcal Q _t} D(u) : D(w) 
+ \int_0^t \int_{0}^{L} \left( \beta |\partial_{xx} h|^2 + \alpha |\partial_{xxx} h|^2\right)\nonumber\\
+\left[ \int_{0}^{L} \left( \frac{\gamma}{2} |\partial_{xx} h|^2 -  \rho_s\partial_{t} h \, \partial_{xx} h \right) \right]_{s=0}^{s=t}  =\int_0^t\int_0^L  \rho_s|\partial_{tx} h|^2.\label{eq_testw}
\end{eqnarray}
We first show that this identity leads to an estimate that is comparable to \eqref{est_dist} up to remainder
terms we shall bound afterwards.  The term that will enable us to bound $h^{-1}$  is $2 \mu \int_{\mathcal Q _t} D(u) : D(w) $.
 To deal with this term, we introduce a well chosen pressure
\begin{equation}
\label{def_q}
q(x,y,t) := q_s(x,t) + \partial_{xy}\psi(x,y,t) \,, \quad q_s(x,t) := - \int_{0}^{x} \partial_{yyy} \psi(s,y,t) {\rm d}s, \quad \forall \, (x,y,t) \in \mathcal Q_T\,.
\end{equation}
An easy computation gives
$\displaystyle\partial_{yyy} \psi(s,y,t) = - 12 \frac{\partial_x h(x, t)}{(h(x,t))^3},$
so that $q_s$ satisfies
\begin{equation}
\label{def_qs}
q_s(x,t) = 12 \int_0^{x} \dfrac{\partial_x h(s,t)}{(h(s,t))^3}{\rm d}s = 6 \left[  \dfrac{1}{|h(0,t)|^2}- \dfrac{1}{|h(x,t)|^2} \right] \,, \quad \forall \, (x,t) \in (0,L) \times (0,T)\,. 
\end{equation}
In particuliar, $q_s$ does not depend on $y.$  Furthermore $\nabla q\in L^2( \mathcal Q_T)$. Applying again the fact that $w$ is divergence-free,
we obtain:
\begin{eqnarray}
2 \int_{\mathcal Q _t} D(u) : D(w) &=& \int_{\mathcal Q _t} (2 D(w) - q  I_2) : D(u) \nonumber\\ 
							&=& \int_0^t \int_{\partial \mathcal F (s)} (2D(w)  - q I_2 ) n \cdot u  - \int_{\mathcal Q _t}  ({\Delta} w - \nabla q ) \cdot u \label{forcevisqueuse}\,. 
\end{eqnarray}
Note, that all the terms make sense thanks to the regularity of $(w, q)$.
In this last identity, by  definition \eqref{def_q} of $q$, we have:
\begin{eqnarray*}
\notag \int_{\mathcal Q _t}  ({\Delta} w - \nabla q ) \cdot u &=& \int_{\mathcal Q _t} \partial_{xxx} \psi \: u_2 - 2 \partial_{yxx} \psi \: u_1 \\
\notag			&=& \int_{0}^t  \int_{\partial \mathcal F(s)}\left( n_1 \partial_{xx} \psi u_2 - 2 n_2 \partial_{xx} \psi u_1 \right){\rm d} \sigma -  \int_{\mathcal Q _t} \partial_{xx} \psi (  \partial_x u_2 -  2\partial_{y} u_1) \\
			&=& -  \int_{0}^t\!\! \int_{0}^{L}\partial_{xx} \psi(x,h(x,s),s)  \partial_{t} h(x,s) \partial_{x} h(x,s){\rm d}x{\rm d}s
			  - \int_{\mathcal Q _t} \partial_{xx} \psi (  \partial_x u_2 -  2\partial_{y} u_1)\,.
\end{eqnarray*}
Similarly, the other term of \eqref{forcevisqueuse}  writes
\begin{multline*}
 \int_0^t \int_{\partial \mathcal F (s)} (2D(w)  - q I_2 ) n \cdot u \,\\
 = \int_{0}^t \int_{0}^{L} \partial_t h(x,t) \Big( (\partial_{yy} \psi(x,h(x,s),s) - \partial_{xx} \psi(x,h(x,s),s)) \partial_{x} h(x,s) \\+  \partial_{yx} \psi(x,h(x,s),s) - q_s(x,s) \Big) {\rm d}x{\rm d}s\,.
\end{multline*}
Differentiating the identity $\partial_y \psi(x,h(x,s), s) = 0$ (that holds true since $\chi_0'(1)=0$), with respect to $x$, yields
$$
\partial_{yx} \psi(x,h(x,s),s) + \partial_x h(x,s) \partial_{yy} \psi(x,h(x,s), s) = 0. 
$$
Consequently, we simplify:
\begin{multline*}
 \int_0^t \int_{\partial \mathcal F (s)} (2D(w)  - q I_2 ) n \cdot u \,
 = - \int_{0}^t \int_{0}^{L} \partial_t h(x,t)\partial_{xx} \psi(x,h(x,s),s)) \partial_{x} h(x,s) {\rm d}x{\rm d}s  \\ -   \int_{0}^t \int_{0}^{L}  \partial_t h(x,s) q_s(x,s)  {\rm d}x{\rm d}s\,.
\end{multline*}
Combining the computations of both terms in \eqref{forcevisqueuse}, we obtain finally:
\begin{equation}
\label{term_vis}
2 \int_{\mathcal Q _t} D(u) : D(w)  = - \int_{0}^t \int_{0}^{L} \partial_t h(x,s)q_s(x,s)  {\rm d}x{\rm d}s + \int_{\mathcal Q _t} \partial_{xx} \psi (  \partial_x u_2 -  2\partial_{y} u_1).
\end{equation}
At this point, we replace $q_s$ by its explicit value (see \eqref{def_qs}) and  by remembering that the average of $\partial_t h$  is zero, we obtain 
\begin{equation}
\label{term_p}
 \int_{0}^t \int_{0}^{L} \partial_t h(x,s)q_s(x,s)  {\rm d}x{\rm d}s  = -6\int_0^t \int_{0}^{L}  \dfrac{\partial_t h(x,s)}{|h(x,s)|^2}  {\rm d}x{\rm d}s = \left[ \int_{0}^{L} \dfrac{6}{h(x,s)} {\rm d}x\right]_{s=0}^{s=t}\,.
\end{equation}
Consequently, from \eqref{term_vis}, \eqref{term_p}, the equality \eqref{eq_testw} reduces to:
\begin{multline} \label{eq_testw_red}
 \left[ \int_{0}^{L} \left( \frac{\gamma}{2} |\partial_{xx} h|^2 -  \rho_s\partial_{t} h \, \partial_{xx} h + \dfrac{6\mu}{h} \right)  \right]_{s=0}^{s=t}
  + \int_0^t \int_{0}^{L} \left( \beta |\partial_{xx} h|^2 + \alpha |\partial_{xxx} h|^2 \right) \\
= 
\int_0^t \int_{0}^{L} \rho_s |\partial_{tx} h|^2 { +  \mu \int_{\mathcal Q _t} \partial_{xx} \psi (  \partial_x u_2 -  2\partial_{y} u_1)
+ \int_{\mathcal Q _t} \rho_f(\partial_{t} u + u \cdot \nabla u ) \cdot w}
\end{multline}
We recognize in the left-hand side of this equality the quantities that we want to estimate as in \eqref{est_dist}. 
Compared to \eqref{est_dist}, we have two additional terms 
$$ 
 T_1=\mu \int_{\mathcal Q _t} \partial_{xx} \psi (  \partial_x u_2 -  2\partial_{y} u_1)\,, \quad T_2= \int_{\mathcal Q _t} \rho_f(\partial_{t} u + u \cdot \nabla u ) \cdot w\,.
 $$
To bound these terms, we need precise estimates on the stream-function $\psi$ that are gathered in Appendix \ref{app_psi}.

\medskip

First, $T_1$ is bounded by applying  Proposition \ref{propo:psi} and energy estimate \eqref{decay_ec}
\begin{eqnarray*}
 \left| \int_{\mathcal Q _t} \partial_{xx} \psi (  \partial_x u_2 -  2\partial_{y} u_1)\right| & \leq &  C\|\partial_{xx} \psi \: ;\:  L^2(\mathcal Q _t)\| \|\nabla u \:; \: L^2(\mathcal Q _T)\|\,, \\ 
&\leq&  C_0 \Bigl( \int_{0}^t \Bigl[  \|h ; L^{\infty}_{\sharp}(0, L)\| \|\partial_{xxx} h ; L^2_{\sharp}(0, L)\|^2  \\
&& +  \|\partial_{xx} h ; L^2_{\sharp}(0, L)\|^{\frac{3}{2}}\|\partial_{xxx}h ; L^2_{\sharp}(0, L)\|^{\frac{3}{2}}\Bigr]\,\Bigr)^{\frac{1}{2}}, \notag
\end{eqnarray*}
with a constant $C_0$ depending only on the initial data.  Once again, using the energy estimate \eqref{decay_ec}, we obtain
\begin{eqnarray}
\notag \left| \int_{\mathcal Q _t} \partial_{xx} \psi (  \partial_x u_2 -  2\partial_{y} u_1)\right| & \leq& C_0\Biggl[ \left( \sup_{t \in (0,T)}\| h ; L^{\infty}_{\sharp}(0,L)\|^{\frac{1}{2}}\right)  \left( \int_{0}^t \|\partial_{xxx} h ; L^2_{\sharp}(0,L)\|^2\right)^{\frac{1}{2}}  \\ 
\notag && \quad+  \left( \sup_{t \in (0,T)}\|\partial_{xx} h ; L^2_{\sharp}(0,L)\|^{\frac{3}{4}}\right)  \left( \int_{0}^t  \|\partial_{xxx} h ; L^2_{\sharp}(0, L)\|^{\frac{3}{2}}\right)^{\frac{1}{2}}  \Biggr] \\
\label{est2_tech0}&\leq& C_{0} (1+T)  + \varepsilon_1  \int_{0}^t \|\partial_{xxx} h ; L^2_{\sharp}(0, L)\|^2\,,
\end{eqnarray}
for arbitrary small $\varepsilon_1 >0$, that will be choosen later on. We note that $C_0$ depends on $\varepsilon_1$
{\em a priori} but the value of this parameter will be fixed to a universal constant afterwards. This remark is also valid when other $\varepsilon$'s that are introduced.

\medskip

Concerning $T_2$, taking into account the convection of the fluid domain by the fluid velocity, we have, 
\begin{equation}
\label{T2}
T_2=\int_{\mathcal Q _t} (\partial_{t} u + u \cdot \nabla u ) \cdot w = \left[ \int_{\mathcal F (s)} u(\cdot,s) \cdot w(\cdot,s) \right]_{s=0}^{s=t} - \int_{\mathcal Q _t} (\partial_t w + u \cdot \nabla w) \cdot u\,. 
\end{equation}
We bound the first term on the right-hand side  of \eqref{T2} using a Cauchy-Schwarz inequality.  Applying the energy estimate \eqref{decay_ec} and  the estimates \eqref{eq_dxpsi}-\eqref{eq_dypsi} on the gradient of the stream-function,  one gets
\begin{eqnarray} \notag
\left| \left[ \int_{\mathcal F (s)} u(\cdot,s) \cdot w(\cdot,s) \right]_{s=0}^{s=t}\right| &\leq& C_0 +   \|u(\cdot,t) ; L^2(\mathcal F(t)) \|  \|\nabla \psi(\cdot,t) ; L^2(\mathcal F(t)) \|\,\\
\notag & \leq & C_0 \left( 1  +  \left[ \int_{0}^{L} \dfrac {{\rm d}x}{h(x,t)}   \right]^{\frac{1}{4}} \right)\\
\label{est2_tech1}& \leq &{C_0}+ \varepsilon_2   \int_{0}^{L} \dfrac {{\rm d}x}{h(x,t)} \,, 
\end{eqnarray}
for arbitrary small $\varepsilon_2 >0.$ 

\medskip

For the second term of the right hand side of \eqref{T2}, we first integrate by parts in space. Since $\psi(x, h(x, s), s) =\partial_x h(x, s)$ and $\partial_y \psi(x,h(x,s),s) = 0,$ we have $\partial_t \psi(x,h(x,s),s) = \partial_{tx} h(x,s)$
so that:
\begin{eqnarray*}
\int_{\mathcal Q _t} u \cdot \partial_t w &=& - \int_{0}^t \int_{0}^{L} \partial_t h(x,s)  \partial_x h(x,s) \partial_t \psi(x,h(x,s),s){\rm d}x {\rm d}s + \int_{\mathcal Q _t}  (\partial_y u_1-\partial_x u_2 ) \partial_t\psi\\
&=&- \int_{0}^t \int_{0}^{L} \partial_t h(x,s) \partial_x h(x,s) \partial_{tx} h(x, s) {\rm d}x {\rm d}s + \int_{\mathcal Q _t}  (\partial_y u_1-\partial_x u_2 ) \partial_t\psi.
\end{eqnarray*}
Applying the energy estimate \eqref{decay_ec} and using the following 1D embedding inequality 
$$
\|\partial_{x} h ; L^{\infty}_{\sharp}(0, L)\| \leq C \|\partial_{xx} h ; L^{2}_{\sharp}(0, L)\|,
$$
we can estimate the boundary term
\begin{eqnarray}
\notag \left| \int_{0}^t \int_{0}^{L} \partial_t h(x,s) \partial_x h(x,s) \partial_{tx} h(x, s) {\rm d}x {\rm d}s\right|
&\leq& \int_{0}^t  \|\partial_{t} h ; L^2_{\sharp}(0, L)\| \|\partial_{x} h ; L^{\infty}_{\sharp}(0, L)\|  \|\partial_{tx} h ; L^2_{\sharp}(0,L)\| \\
\notag & \leq & C\int_{0}^t  \|\partial_{xx} h ; L^{2}_{\sharp}(0,L)\|  \|\partial_{tx} h ; L^2_{\sharp}(0,L)\|^2 \\[4pt]
\label{est2_tech2}&\leq& C_0 \,. 
\end{eqnarray}
Taking into account \eqref{eq_dtpsi} and the energy estimate \eqref{decay_ec}, we now bound the second term by
\begin{eqnarray}
\notag \left| \int_{\mathcal Q _t}  (\partial_y u_1- \partial_x u_2) \partial_t\psi\right|
&\leq &  C_0 \|\partial_t \psi ; L^2(\mathcal Q_t)\|\,\\
\notag & \leq & C_0 \left[ \int_0^t  \left( \|\partial_{tx} h ; L^2_{\sharp}(0,L)\|^2 + \|\partial_{xxx} h ; L^2_{\sharp}(0,L)\|\right)\right]^{\frac 12}\\
\label{est2_tech3}& \leq & C_{0}(1+T)+ \varepsilon_3 \int_0^t  \|\partial_{xxx} h ; L^2_{\sharp}(0,L)\|^2 \,,
\end{eqnarray}
where $\varepsilon_3>0$  will be chosen later on.

\medskip

Finally for the last term in the right hand side of \eqref{T2}, we have, after space integration by parts:
\begin{equation}
\label{dist:conv}
\int_{\mathcal Q _t} u \cdot \nabla w \cdot u 
= \int_0^t \int_{0}^{L} |\partial_t h(x,s)|^2 \partial_{xx} h(x,s) {\rm d}x {\rm d}s - \int_{\mathcal Q _t} u \cdot \nabla u \cdot w 
\,,\\
\end{equation}

Concerning the boundary integral in \eqref{dist:conv}, we apply \eqref{decay_ec} to show the following estimate
\begin{eqnarray}
\notag
 \left|\int_0^t \int_{0}^{L} |\partial_t h(x,s)|^2 \partial_{xx} h(x,s){\rm d}x {\rm d}s\right| &  \leq &  C \int_0^t \left( \|\partial_{xx} h(\cdot,s) ; L^2_{\sharp}(0,L)\| \|\partial_t h(\cdot,s) ; L^{\infty}_{\sharp}(0,L)\|^2 \right){\rm d}s
\\ &\leq& \sup_{t \in (0,T)} \|\partial_{xx} h ; L^2_{\sharp}(0,L)\| \quad \int_0^t \|\partial_{tx} h ; L^2_{\sharp}(0,L)\|^2 \leq C_0 \,.\label{boundary_T2}  
\end{eqnarray}
Moreover for the volume integral in the right hand side of \eqref{dist:conv}, we have
\begin{eqnarray*}
\left|\int_{\mathcal Q _t} u \cdot \nabla u \cdot w  \right|
& \leq & \int_0^t \int_{0}^{L} \left(\left( \int_{0}^{h(x,s)} |u|^2\right)^{\frac{1}{2}} \left(\int_{0}^{h(x,s)} |\nabla u|^2\right)^{\frac{1}{2}} \sup_{y \in (0,h(x,s))} |w(x,y,s)|\right){\rm d}x {\rm d}s.
\end{eqnarray*}
From  \eqref{eq_nablapsi}  we know that the following pointwise estimate holds true
$$
|w(x,y,s)| \leq C\left( |\partial_{xx} h(x,s)| + \dfrac{|\partial_{x} h(x,s)|}{h(x,s)}  +  \dfrac{|\partial_{x} h(x,s)|^2}{h(x,s)}\right) \,, \quad \forall \, (x,y) \in \mathcal F(s).
$$
Thus we define $$I_1=\int_0^t \int_{0}^{L} \left(\left( \int_{0}^{h(x,t)} |u|^2\right)^{\frac{1}{2}} \left(\int_{0}^{h(x,s)} |\nabla u|^2\right)^{\frac{1}{2}}   |\partial_{xx} h(x,s)|\right){\rm d}x {\rm d}s,$$ 
$$I_2=\int_0^t \int_{0}^{L} \left(\left( \int_{0}^{h(x,s)} |u|^2\right)^{\frac{1}{2}} \left(\int_{0}^{h(x,s)} |\nabla u|^2\right)^{\frac{1}{2}} \dfrac{|\partial_{x} h(x,s)|}{h(x,s)}\right){\rm d}x {\rm d}s$$
 and  
$$I_3=\int_0^t \int_{0}^{L} \left(\left( \int_{0}^{h(x,s)} |u|^2\right)^{\frac{1}{2}} \left(\int_{0}^{h(x,s)} |\nabla u|^2\right)^{\frac{1}{2}}  \dfrac{|\partial_{x} h(x,s)|^2}{h(x,t)}\right){\rm d}x {\rm d}s.$$

We now take care of each quantity.
Applying the 1D embedding inequality 
$$
\|\partial_{xx} h ; L^{\infty}_{\sharp}(0,L)\| \leq C \|\partial_{xxx} h ; L^2_{\sharp}(0,L)\|
$$
and the energy estimate \eqref{decay_ec}, we obtain:
\begin{eqnarray*}
I_1& \leq &  C\int_0^t \left( \left( \int_{\mathcal F (s)}|u|^2\right)^{\frac{1}{2}} \left(\int_{\mathcal F (s)} |\nabla u|^2\right)^{\frac{1}{2}}   \|\partial_{xx} h ; L^{\infty}_{\sharp}(0,L)\|   \right){\rm d}s \\
      & \leq & C_0 \int_0^t\left(\left(\int_{\mathcal F (s)} |\nabla u|^2\right)^{\frac{1}{2}}  \|\partial_{xxx} h ; L^{2}_{\sharp}(0,L)\|   \right){\rm d}s\\
      & \leq& C_0 \| \nabla u ; L^2_{\sharp}(\mathcal Q_T)\| \left[ \int_0^t \|\partial_{xxx} h ; L^2_{\sharp}(0,L)\|^2 {\rm d}s \right]^{\frac 12} \\
      & \leq & \dfrac{C_0}{\varepsilon_4} + \varepsilon_4  \int_0^t \|\partial_{xxx} h ; L^2_{\sharp}(0,L)\|^2 {\rm d}s\,,
\end{eqnarray*}
for arbitrary small $\varepsilon_4 >0.$ 
We now take care of $I_2$. We note then that $u$ vanishes on $y=0$ so that  a Poincar\'e inequality yields:
$$
\left( \int_{0}^{h(x,s)} |u|^2\right)^{\frac{1}{2}} \leq C  h(x,s) \left( \int_{0}^{h(x,s)} |\nabla u|^2\right)^{\frac{1}{2}}.
$$
Using this bound to estimate $I_2$ we get 
\begin{eqnarray*}
I_2 
& \leq & C  \int_0^t \| \nabla u ; L^2_{\sharp}(\mathcal F(s))\|^2  \|{\partial_x h} ; L^{\infty}_{\sharp}(0,L)\|  \, {\rm d}s \leq C_0
\end{eqnarray*}
as $\|{\partial_x h} ; L^{\infty}_{\sharp}(0,L)\| \leq C\|{\partial_{xx} h} ; L^2_{\sharp}(0,L)\|$ which remains uniformly bounded in time (see \eqref{decay_ec}).
With similar arguments, we also prove $I_3 \leq C_0.$ This yields finally 

\begin{eqnarray*}
\left|\int_{\mathcal Q _t} u \cdot \nabla u \cdot w  \right|
& \leq &  C_0 + \varepsilon_4 \int_0^t \|\partial_{xxx} h ; L^2_{\sharp}(0,L)\|^2 {\rm d}s  
\end{eqnarray*}
and thus, taking into account \eqref{boundary_T2}
\begin{equation} \label{est2_tech4}
\left| \int_{\mathcal Q _t} u \cdot \nabla w \cdot u \right| \leq  C_0 + \varepsilon_4 \int_0^t \|\partial_{xxx} h ; L^2_{\sharp}(0,L)\|^2 {\rm d}s\,.   
\end{equation}
Finally, $T_2$ can be bounded, thanks to \eqref{est2_tech1}, \eqref{est2_tech2}, \eqref{est2_tech3}, \eqref{est2_tech4}, as
\begin{equation}
\label{est_T2}
|T_2|\leq C_{0}(1+T) +\varepsilon_5 \int_0^t   \|\partial_{xxx} h ; L^2_{\sharp}(0,L)\|^2 + \varepsilon_2   \int_{0}^{L} \dfrac {{\rm d}x}{h(x,t)} ,
\end{equation}
where $\varepsilon_5$ and $\varepsilon_2$ are to be chosen small enough.

\medskip

Combining \eqref{est_T2} and \eqref{est2_tech0} to bound the right-hand side of \eqref{eq_testw_red} and taking into account
\eqref{decay_ec} to bound the remaining terms on the right-hand side depending on $h$ that might be concerned, we get, for any $t\leq T$,
\begin{multline*}
 \int_{0}^{L} \left( \frac{\gamma}{2} |\partial_{xx} h(\cdot,t)|^2 + \dfrac{6\mu}{h(t,\cdot)} \right)
  + \int_0^t \int_{0}^{L} \left( \beta |\partial_{xx} h|^2 + \alpha |\partial_{xxx} h|^2 \right) \\[8pt]
\begin{array}{rcl}
&\leq& 
C_{0 }(1+T)+ \rho_s \|\partial_t h(\cdot,t) ; L^2_{\sharp}(0, L)\| \| \partial_{xx} h(\cdot,t) ; L^2_{\sharp}(0,L)\| \\[8pt]
&&\quad 
+ \displaystyle{\rho_s\int_0^t \int_{0}^{L}} |\partial_{tx} h|^2 +\varepsilon \left( \int_{0}^{L}\dfrac{6\mu}{h(\cdot,t)}  + \int_0^t \|\partial_{xxx} h ; L^2_{\sharp}\|^2 {\rm d}s\right)\, \\[10pt]
&\leq &C_{0}(1+ T)+\varepsilon \left( \displaystyle\int_{0}^{L}\dfrac{6\mu}{h(\cdot,t)}  + \int_0^t \|\partial_{xxx} h ; L^2_{\sharp}(0, L)\|^2 {\rm d}s\right)\,.
\end{array}
\end{multline*}
for some arbitrary small $\varepsilon >0$. We conclude the proof of the distance estimate  of Proposition \ref{prop_distest} by choosing $\varepsilon$ small enough.

\begin{rem} $\phantom{23}$ \\[-10pt]
\begin{itemize}
\item The  derived distance estimate relies strongly on the fact that the beam motion is only transverse and that we control the curvature of the elastic boundary. Indeed, we need to have bounds on the deformation $h$ in $L^\infty(0, T; H^2_\sharp(0,L))$ and in $L^2(0, T; H^3_\sharp(0,L))$  to control remainder terms. Both norms are bounded because $\alpha > 0,$ even if the first one is controlled via the energy bound  while the second one is controlled simultaneously  with  $h^{-1}$. \\
\item To prove our distance estimate we need to control  $\partial_t \eta$ in $L^2(0, T; H^1_\sharp(0, L)).$ This is the reason why we assume $\gamma>0$. One may wonder whether the fluid dissipation would be sufficient.  A priori, from the $L^2(0, T; H^1_\sharp({\mathcal F}(t))$ bound of the fluid velocity we only get a control on $\partial_{t} h$ in  $L^2(0, T; H^{1/2}_\sharp(0, L))$, which is not enough.
\end{itemize}
\end{rem}

\subsection{Regularity estimate}
Combining Proposition \ref{prop_distest} with Proposition \ref{prop_controlparH2}, we obtain that $h^{-1} \in L^\infty(0, T; L^1_\sharp(0, L))$. 
More precisely, we have that there exists a non-decreasing function $R: [0,T^*) \to [0,\infty)$ bounded on all bounded subintervals of $[0,T^*)$ such that:
\begin{equation} \label{eq_Rt}
\|h^{-1}(\cdot,t) ; L^{\infty}(0, L)\| + \|h(\cdot,t) ; H^2(0,L)\| \leq R_t \quad \forall \, t \in [0,T^*).
\end{equation}
In particular, the beam never touches the bottom of the fluid cavity on bounded time intervals. This lower bound on  $h$ ensures that the elliptic regularity result derived in Section \ref{subsection:elliptique} applies and enables us to pursue further in order to prove global existence of strong solutions by choosing appropriate test functions. We obtain the following quantitative estimate:
\begin{proposition} \label{prop_estreg}
There exists a function $C^0_{\rm reg} : [0,T^*) \to [0,\infty)$ bounded on all bounded subintervals of $[0,T^*)$ such that:
\begin{equation}
\|u(\cdot,t); H^1_{\sharp}(\mathcal F(t))\|^2 + \|\partial_t h(\cdot,t) ; H^1_{\sharp}(0,L)\|^2 + \|h(\cdot,t) ; H^3_{\sharp}(0,L)\|^2  \leq  C^0_{\rm reg}(t)
\end{equation}
\end{proposition}
The remainder of this subsection  is devoted to the proof of this proposition.  We fix $T<T^*$ and construct the $C^0_{\rm reg}$ for $t \in [0,T].$
We split the proof into three steps. In the first one,  we aim at multiplying  the fluid equation and the structure equation by $\partial_t u$ and $\partial_{tt} h$ respectively. Yet, $\partial_t u$  is not an appropriate multipier of the fluid equation  since it does not take into account the motion of the fluid domain. 
A natural choice is then the total derivative $\partial_t u+ u \cdot \nabla u$, but this function is not divergence free.  So we  introduce a modified divergence free test function following ideas of \cite{Cum-Tak07}. This step requires to bound the fluid velocity in $H^2_\sharp({\mathcal F}(t))$ and consequently, thanks to the elliptic estimates derived in Section \ref{subsection:elliptique}, the structure velocity in $H_\sharp^{{3/2}}(0, L)$.  As we do not get estimates on this quantity in our first step, we need a second step to obtain a regularity estimate on the deformation of the beam.  This second estimate depends itself on the regularity of the applied fluid-force and thus on high-order norms of the fluid-velocity and the pressure. The final step consists in a well chosen combination of the two previous estimates in order to obtain the expected result.

\medskip 

\begin{rem}$\phantom{123}$ \\[-14pt] \label{rem_constantes}
\begin{itemize}
\item  As the function $R:[0,\infty) \mapsto (0,\infty)$ satisfies \eqref{eq_Rt} at-hand, when we apply below Proposition \ref{prop_cdv}, Lemma \ref{lem_ellest} or Corollary \ref{cor_phi}, the associated respective constants $K^e,K^s$ or $K^b$ 
define non-decreasing functions of time bounded on bounded subintervals of $[0,T^*).$\\

\item In the computations below, we denote by  $C_0$ a constant depending on inital data and by $C_0:[0,T^*) \rightarrow (0,\infty)$ a function that is bounded on all bounded subintervals of $[0,T^*).$ The value of this constant/function may vary between lines.  
\end{itemize}

\end{rem}

\medskip

\paragraph{\em Step 1 : $L^\infty(0, T; H^1)$ regularity estimate of the fluid and structure velocities }
As explained in the previous paragraph, we introduce a vector-field $\hat{\Lambda}(\cdot,t)$ which coincides with $u$ on $\partial{\mathcal F(t)}$ for all $t \in (0,T).$  and which satisfies
\begin{itemize}
\item  $\hat{\Lambda}$ is divergence free,
\item  $\nabla \hat{\Lambda} \in  L^2_{\sharp}(\mathcal Q_T)$ and $\nabla^2 \hat{\Lambda} \in  L^2_{\sharp}(\mathcal Q_T) $ with
\begin{eqnarray}
\label{eq_hatLambda0}
\| \hat{\Lambda}(\cdot,t) ; L^2_{\sharp}(\mathcal F(t))\| &\leq& C_0(t) \|\partial_{t} h(\cdot,t) ;  L^2_{\sharp}(0,L)\|\, , \\
 \label{eq_hatLambda1}
\|\nabla \hat{\Lambda}(\cdot,t) ; L^2_{\sharp}(\mathcal F(t))\| &\leq& C_0(t) \|\partial_{t} h(\cdot,t) ;  H^{1}_{\sharp}(0,L)\|\, , \\
\label{eq_hatLambda2}
 \|\nabla^2 \hat{\Lambda}(\cdot,t) ; L^2_{\sharp}(\mathcal F(t))\| &\leq&  C_0(t) 
 \|\partial_{t} h(\cdot,t) ;  H^{2}_{\sharp}(0,L)\|\,, 
 \end{eqnarray}
 for a.e. $t \in (0,T),$  
\item $\hat{\Lambda} = u$ and $\partial_2 \hat{\Lambda} = 0$ on $y = h(x, t)$, 
\item $\hat{\Lambda} = 0$ on $y=0\,.$
\end{itemize}
The construction of $\hat\Lambda$ is given in Appendix \ref{App:reg}. With the notations of this appendix, 
we have $C_0(t) = K^l(R_t)$ that is indeed a function which is bounded on all bounded subintervals of $[0,T^*).$

\medskip

Next we define $v$ as
$$
v := \partial_t u + \hat{\Lambda} \cdot \nabla u - u \cdot \nabla \hat{\Lambda}.
$$
Given $t \leq T,$ it is a  suitable multiplier for \eqref{eq_NS} on $\mathcal Q_t$ as it belongs to $L^2_{\sharp}(\mathcal Q_t)$. Indeed,  by classical Sobolev embedding, we have
\begin{eqnarray*}
\|\hat{\Lambda} \cdot \nabla u ; L^2(\mathcal Q_t)\| &\leq&  \int_0^t \|\hat{\Lambda} ; L^4 (\mathcal F(s))\| \| \nabla u ; L^4(\mathcal F(s))\| \\
& \leq & C _0(t) \int_0^t  \|\hat{\Lambda} ; H^1 (\mathcal F(s))\| \|  u ; H^2(\mathcal F(s))\|\, \\
&  \leq & C_{0}(t)  \sup_{s \in (0,t)} \|\partial_t h(\cdot,s); H^1_{\sharp}(0,L)\| \, \left( \| \nabla u ; L^2(\mathcal Q_t) \| + \|\nabla^2 u ; L^2(\mathcal Q_t)\|\right)\,.
\end{eqnarray*}
We note here that we have used the continuous embedding $L^4(\mathcal F(s)) \hookrightarrow H^1(\mathcal F(s)).$
A priori,  the constant associated with this embedding depends on the domain $\mathcal F(s)$ and thus on the deformation of the beam. However, going back in a fixed domain and interpolating the results of Proposition \ref{prop_cdv} (or extrapolating an equivalent version for the $L^4$-space that we skip for conciseness), we might prove that this constant is uniformly bounded locally in time, as  $\|h(\cdot,s) ; H^2_{\sharp}(0, L)\|$ and $\|h^{-1}(\cdot,s) ; L^{\infty}_{\sharp}(0, L)\|$ are bounded locally.  In the same way, in what follows,  we may use interpolation inequalities in $\mathcal F(s)$ for which the constant will be bounded by a locally bounded function of $\|h(\cdot,s) ; H^2_{\sharp}(0, L)\|$ and $\|h^{-1}(\cdot,s) ; L^{\infty}_{\sharp}(0, L)\|$ and thus by a locally bounded function of $R_s	$. Similarly, we obtain $u \cdot \nabla \hat{\Lambda} \in L^2(\mathcal Q_t)$. 

\medskip

So, we multiply \eqref{eq_NS} by $v$ on $\mathcal Q_t$ and obtain the following identity
\begin{equation}
\label{eq_mult_fluid}
\int_{\mathcal F(t)} \rho_f(\partial_t u + u \cdot \nabla u  ) \cdot (\partial_t u + \hat{\Lambda} \cdot \nabla u - u \cdot \nabla \hat{\Lambda} )
= 
\int_{\mathcal F(t)} {\rm div} \sigma \cdot   (\partial_t u + \hat{\Lambda} \cdot \nabla u - u \cdot \nabla \hat{\Lambda} )\,.
\end{equation}
Formally, the trace of $v$ on $y=h$ is equal to $\partial_{tt}h$. So we expect that this identity has to be combined by the structure equation multiplied
by $\partial_{tt}h.$ Again, $\partial_{tt} h$ is a suitable multiplier for \eqref{eq_elasticite} as it belongs to $L^2_{\sharp}((0,L) \times (0,t))$ 
and we obtain the following identity:
\begin{equation} \label{eq_mult_beam}
\int_{0}^{L} \rho_s|\partial_{tt} h|^2 -\int_{0}^{L} \phi(u,p,h) \partial_{tt} h  =-\alpha \int_0^L \partial_{xxxx}h \partial_{tt} h + \beta\int_0^L \partial_{xx}h \partial_{tt} h +\gamma\int_0^L \partial_{xxt}h \partial_{tt} h 
\end{equation}
To compute the right-hand sides of \eqref{eq_mult_fluid} and \eqref{eq_mult_beam}, we need more regularity than the one satisfied by the strong solution under consideration. Yet the following lemma holds true:
\begin{lemma} 
\label{lem:vitesse}
For any triplet $(w, q, b)$ satisfying the regularity assumptions of Definition \ref{def_strongsolution},  namely
$$
b \in H^2(0,T ; L^2_{\sharp}(0, L)) \cap L^2(0,T; H^4_{\sharp}(0, L)), 
$$
$$
w\in H^1_{\sharp}(\mathcal Q_T)\,,  \quad \nabla^2 w \in L^2_{\sharp}(\mathcal Q_T)\,, \quad
q \in L^2_{\sharp}(\mathcal Q_T)\,, \quad \nabla q\in L^2_{\sharp}(\mathcal Q_T)\,,
$$
and such that $w(x, h(x,t), t)= \partial_t b(x,t) \in L^2_{\sharp,0}(0,L)$ and $\div w=0$, the following identities are satisfied:
\begin{multline}
\label{egalite_vitesse_fluid}
\int_{\mathcal Q_t} {\rm div} \sigma(w, q) \cdot   (\partial_t w + \hat{\Lambda} \cdot \nabla w - w \cdot \nabla \hat{\Lambda} )=- \int_0^t\int_{0}^{L} \phi(w,q,h) \partial_{tt} b\\
-\dfrac \mu2  \int_{\mathcal F(t)} |\nabla w|^{2} +\dfrac \mu2 \int_{\mathcal F(0)} |\nabla w|^{2}(0) - 2\mu\int_{\mathcal  Q_t} D(w) : \left( [\nabla \hat{\Lambda}]^{\top} \nabla w + \nabla \hat{\Lambda} [ \nabla w ]^{\top} - D(w\cdot \nabla \hat{\Lambda}) \right),
\end{multline}
and
\begin{multline}
\label{egalite_vitesse_beam}
\alpha \int_0^t\int_0^L \partial_{xxxx}b \, \partial_{tt} b - \beta\int_0^t\int_0^L \partial_{xx}b \,\partial_{tt} b -\gamma\int_0^t\int_0^L \partial_{xxt}b\, \partial_{tt} b =\\
 \int_{0}^{L} \left(\dfrac{\gamma}{2} |\partial_{tx} b(x, t)|^2 - \beta \partial_{t} b(x,t) \partial_{xx} b(x,t) - \alpha \partial_{tx} b(x,t) \partial_{xxx} b (x,t) \right)\mathrm{d}x
\\
- \beta \int_0^t\int_{0}^{L} |\partial_{tx} b |^2 - \alpha\int_0^t \int_{0}^{L} |\partial_{txx} b |^2
\\ -\int_{0}^{L} \left(\dfrac{\gamma}{2} |\partial_{tx} b(x, 0)|^2 - \beta \partial_t b(x, 0)
\partial_{xx} b(x, 0)    - \alpha \partial_{tx}b(x,0)\partial_{xxx}  b(x,0)\right)\mathrm{d}x  .
\end{multline}
\end{lemma}
 The proof of this lemma relies on regularization arguments and is postponed to Appendix \ref{App:reg}.
 
\medskip

Thus, we apply  \eqref{egalite_vitesse_fluid} for $w=u$ and $b=h$.  Then identity \eqref{eq_mult_fluid} becomes 
\begin{multline}
\label{eq_mult_fluid2}
\int_{\mathcal Q_t} \rho_f(\partial_t u + u \cdot \nabla u  ) \cdot (\partial_t u + \hat{\Lambda} \cdot \nabla u - u \cdot \nabla \hat{\Lambda} )
= - \int_0^t\int_{0}^{L} \phi(u,p,h) \partial_{tt} h\\
-\dfrac \mu2 \int_{\mathcal F(t)} |\nabla u|^{2} +\dfrac \mu2  \int_{\mathcal F(0)} |\nabla u_0|^{2} -2\mu\int_{\mathcal  Q_t} D(u) :  \left( [\nabla \hat{\Lambda}]^{\top} \nabla u + \nabla \hat{\Lambda} [ \nabla u ]^{\top} - D(u\cdot \nabla \hat{\Lambda}) \right),
\end{multline}
For the left-hand side of \eqref{eq_mult_fluid2}, denoted LHS, we have:
\begin{eqnarray*}
LHS &= &\int_{\mathcal Q_t} \rho_f|\partial_t u + u\cdot \nabla u |^2  - \int_{\mathcal Q_t }\rho_f (\partial_t u + u \cdot \nabla u ) \cdot  (u \cdot \nabla u ) \\
&&+ \int_{\mathcal Q_t} \rho_f\left( \partial_t u + u \cdot \nabla u \right) \cdot (\hat{\Lambda} \cdot \nabla u - u \cdot \nabla \hat{\Lambda})\,,
\\
&\geq&
\dfrac{1}{2}  \int_{\mathcal  Q_t} \rho_f |\partial_t u + u\cdot \nabla u |^2  - \frac{1}{2}  \int_{\mathcal Q_t} \rho_f |u\cdot \nabla u |^2 + \int_{\mathcal Q_t} \rho_f \left( \partial_t u + u \cdot \nabla u \right) \cdot (\hat{\Lambda} \cdot \nabla u - u \cdot \nabla \hat{\Lambda})\,,
\end{eqnarray*}
which yields
\begin{eqnarray} \label{eq_estDubase}
&&\dfrac \mu2  \int_{\mathcal F (t)}|\nabla u|^2+ \dfrac 12  \int_{\mathcal  Q_t} \rho_f|\partial_t u + u \cdot \nabla u|^2  + \int_0^t\int_{0}^{L} \phi(u,p,h) \partial_{tt} h 
\\
&& \leq \dfrac 12 \int_{\mathcal Q_t}\rho_f |u\cdot \nabla u|^2- \int_{\mathcal Q_t} \rho_f(\partial_t u + u \cdot \nabla u) (\hat{\Lambda} \nabla u - u \cdot \nabla \hat{\Lambda}) \notag \\ 
&& \quad
- 2\mu  \int_{\mathcal Q_t} D(u) :  \left(  [\nabla \hat{\Lambda}]^{\top} \nabla u + \nabla \hat{\Lambda} [ \nabla u ]^{\top} - D(u\cdot \nabla \hat{\Lambda})\right) +\dfrac \mu2 \int_{\mathcal F(0)} |\nabla u_0|^{2}\,. \notag
\end{eqnarray}
We split the right-hand side of this inequality into six integrals denoted $I_1,\ldots,I_6$ that we bound independently.
Applying interpolation inequalities for estimating the $L^4$-norm, we have
\begin{eqnarray*}
I_1  &:=&  \int_{\mathcal Q_t} \rho_f | u \cdot \nabla u|^2  \\[4pt]
	&\leq &  C_0(t)\int_0^t \|u ; L_{\sharp}^4(\mathcal F(s))\|^2 \|\nabla u ; L_{\sharp}^4(\mathcal F(s))\|^2 {\mathrm d}s\\[6pt]
	&  \leq &  C_{0}(t) \,\int_0^t \| u ; L_{\sharp}^2(\mathcal F(s))\| \|\nabla u ; L_{\sharp}^2(\mathcal F(s))\|^2 \| u ; H^2(\mathcal F(s))\|{\mathrm d}s.
\end{eqnarray*}
Next we use the elliptic estimates derived in section \ref{subsection:elliptique} to bound $\| u ; H^2(\mathcal F(s))\|$. 
\begin{equation}
\label{inegalite:ellip}
\| u ; H^2_{\sharp}(\mathcal F(s))\| \leq K^s(R_t) ( \| \partial_t u + u \cdot \nabla u ; L^2_{\sharp}(\mathcal F(s))\| + \|\partial_{t} h ; H^{\frac 32}_{\sharp}(0,L)\|) .
\end{equation}
Following Remark \ref{rem_constantes}, the  function of time $t \mapsto K^s(R_t)$ is (non-decreasing) and bounded on all bounded subintervals of $[0,T^*).$  Thus, there holds: 
$$
I_1  \leq   C_{0}(t) \,\int_0^t \| u ; L_{\sharp}^2(\mathcal F(s))\| \|\nabla u ; L_{\sharp}^2(\mathcal F(s))\|^2\left( \| \partial_t u + u \cdot \nabla u ; L^2_{\sharp}(\mathcal F(s))\| + \|\partial_{t} h ; H^{\frac 32}_{\sharp}(0,L)\|\right) {\mathrm d}s.
$$
Recalling the energy estimate proven in Proposition \ref{prop_estkin}  we obtain,  for arbitrary $\varepsilon >0,$ 
\begin{multline} \label{eq_estI1}
I_1 \leq  C_0(t) \int_0^t \|\nabla u ; L^2_{\sharp}(\mathcal F(s))\|^4 {\mathrm d}s \\ + \varepsilon \int_0^t\left(\| \partial_t u + u \cdot \nabla u ; L^2_{\sharp}(\mathcal F(s))\|^2+ \|\partial_{txx} h ; L^2_{\sharp}(0,L)\|^2\right){\mathrm d}s\,.
\end{multline}
Here and in what follows, we skip the dependance of $C_0(t)$ on $\varepsilon$ as this parameter will be fixed universally later on.
The second term reads:
$$
I_2 := \int_{\mathcal Q_t} \rho_f(\partial_t u + u \cdot \nabla u)  \cdot (\hat{\Lambda} \cdot \nabla u) \,,
$$
that we estimate as follows
\begin{eqnarray*}
|I_2| &\leq& \int_0^t\| \partial _t u  + u \cdot \nabla u ; L^2_{\sharp}(\mathcal F(s))\|  \|\hat{\Lambda} ; L^4_{\sharp}(\mathcal F(s))\| \|\nabla u ; L^4_{\sharp}(\mathcal F(s))\|{\mathrm d}s\, \\
 &\leq&  C_0(t) \int_0^t\| \partial _t u  + u \cdot \nabla u  ; L^2_{\sharp}(\mathcal F(s))\| \|\nabla u ; L^2_{\sharp}(\mathcal F(s))\|^{\frac 12} \ldots  \\ 
 				&& \qquad \ldots \|u ; H^2_{\sharp}(\mathcal F(s))\|^{\frac 12} \|\hat{\Lambda} ; L^2_{\sharp}(\mathcal F(s))\|^{\frac 12} \|\hat{\Lambda} ; H^1_{\sharp}(\mathcal F(s))\|^{\frac 12}{\mathrm d}s \, .
\end{eqnarray*}
Thanks to estimates \eqref{eq_hatLambda0},  \eqref{eq_hatLambda1} satisfied by $\hat \Lambda$ and using Proposition \ref{prop_estkin}  which enables us to bound $\| u ;L^2_{\sharp}(\mathcal F(s))\|$ and $\|\partial_t h ; L^2_{\sharp}(0,L)\|$, and the elliptic estimate \eqref{inegalite:ellip}, we get, for arbitrary $\varepsilon >0$
to be fixed later on:
\begin{multline} \label{eq_estI2}
|I_2|  \leq C_0(t)  \int_0^t\left(  \|\nabla u ; L^2_{\sharp}(\mathcal F(s))\|^4 + \|\partial_{tx} h ; L^2_{\sharp}(0,L)\|^4 \right) {\mathrm d}s\\
  + \varepsilon \int_0^t\left( \| \partial _t u  + u \cdot \nabla u  ; L^2_{\sharp}(\mathcal F(s))\|^2 +  \| \partial_{txx} h; L^2_{\sharp}(0, L)\|^2  \right) {\mathrm d}s\,.
\end{multline}
The third term  reads:
$$
I_3 := \int_{\mathcal Q_t}\rho_f (\partial_t u + u \cdot \nabla u) \cdot (u \cdot \nabla \hat{\Lambda})\,.
$$
We estimate it as follows
\begin{eqnarray*}
|I_3| & \leq & \int_0^t \|\partial_t u + u \cdot \nabla u ; L^2_{\sharp}(\mathcal F(s))\| \| u ; L^4_{\sharp}(\mathcal F(s))\|
\|\nabla \hat{\Lambda} ; L^4_{\sharp}(\mathcal F(s))\|{\mathrm d}s\, \\
& \leq&  C_{0}(t)  \int_0^t\|\partial_t u + u \cdot \nabla u ; L^2_{\sharp}(\mathcal F(s))\| \| u ;L^2_{\sharp}(\mathcal{F}(s))\|^{\frac 12} \ldots \\ 
&& \qquad \ldots \|\nabla u ;L^2_{\sharp}(\mathcal F(s))\|^{\frac 12} \|\hat{\Lambda} ; H^1_{\sharp}(\mathcal F(s))\|^{\frac 12}
\|\hat{\Lambda} ; H^2_{\sharp}(\mathcal F(s))\|^{\frac 12}{\mathrm d}s\,,
\end{eqnarray*}
Applying estimates \eqref{eq_hatLambda1}-\eqref{eq_hatLambda2} satisfied by $\hat \Lambda$,  together with interpolation  inequalities and Proposition \ref{prop_estkin}, we obtain
$$
|I_3|  \leq C_{0}(t)  \int_0^t\|\partial_t u + u \cdot \nabla u ; L^2_{\sharp}(\mathcal F(s))\| 
 \|\nabla u ;L^2_{\sharp}(\mathcal F(s))\|^{\frac 12} \|\partial_{tx} h;L^2_{\sharp}(0,L)\|^{\frac 12} 
  \|\partial_{txx} h ; L^2_{\sharp}(0, L)\|^{\frac 12}{\mathrm d}s.
 \,.
$$
Hence, for arbitrary $\varepsilon >0,$ there holds:
\begin{multline}  \label{eq_estI3}
|I_3| \leq C_0(t) \int_0^t\left(  \|\nabla u ; L^2_{\sharp}(\mathcal F(s))\|^4 + \|\partial_{tx} h ; L^2_{\sharp}(0,L)\|^4\right) {\mathrm d}s  \\ + \varepsilon\int_0^t \left( \| \partial _t u  + u \cdot \nabla u  ; L^2_{\sharp}(\mathcal F(s))\|^2 + \| \partial_{txx} h ; L^2_{\sharp}(0,L)\|^2  \right){\mathrm d}s   \,.
\end{multline}
We proceed with 
$$
I_4 := 2\mu\int_{\mathcal Q_t} D(u) : ([\nabla \hat{\Lambda}]^{\top} \nabla u + \nabla \hat{\Lambda} [ \nabla u ]^{\top}),
$$
that we bound as follows ($C$ is a constant depending only on $\mu$):
\begin{eqnarray*}
|I_4|  &\leq&  C \int_0^t\|\nabla u ; L^4_{\sharp}(\mathcal F(s))\|^2 \|\nabla \hat{\Lambda} ; L^2_{\sharp}(\mathcal F(s))\|{\mathrm d}s  \\
		&\leq&  C_{0}(t) \int_0^t\|\nabla u ;L^2_{\sharp}(\mathcal F(s))\| \| u ; H^2_{\sharp}(\mathcal F(s))\| \|\hat{\Lambda} ; H^1_{\sharp}(\mathcal F(s))\|{\mathrm d}s \,.
\end{eqnarray*}
Similarly as above, using \eqref{eq_hatLambda1} and \eqref{inegalite:ellip}, we obtain, for arbitrary $\varepsilon >0$:
\begin{multline}
|I_4| \leq \displaystyle C_0(t) \int_0^t\left( \|\partial_{tx} h ; L^2_{\sharp}(0,L)\|^4 + \|\nabla u ; L^2(\mathcal F(s))\|^4 \right){\mathrm d}s \\
\displaystyle+ \varepsilon \int_0^t\left( \| \partial _t u  + u \cdot \nabla u  ; L^2_{\sharp}(\mathcal F(s))\|^2 + \| \partial_{txx} h ; L^2_{\sharp}(0,L)\|^2  \right){\mathrm d}s \,.\label{eq_estI4}
\end{multline} 
Finally, the fifth term is defined by
$$
I_5 := 2\mu\int_{\mathcal Q_t} D(u) : D(u \cdot \nabla \hat{\Lambda})\,.
$$
Expanding $D(u \cdot \nabla \hat{\Lambda})$, we obtain
$$
|I_5|  \leq C  \left( \int_{\mathcal Q_t} |\nabla u|^2 |\nabla \hat{\Lambda}| + \int_{\mathcal Q_t} |\nabla u| |u ||\nabla^2 \hat{\Lambda}|\right)\,.
$$
The first term on the right-hand side is bounded as $I_4$ (see \eqref{eq_estI4}). As for the second term of the right-hand side, we have
\begin{multline*}
 \int_{\mathcal Q_t} |\nabla u| |u| |\nabla^2 \hat{\Lambda}|
  \leq C_{0}(t) \displaystyle\int_0^t \|\hat{\Lambda} ; H^2_{\sharp}(\mathcal F(s))\| \|u ; L^2_{\sharp}(\mathcal F(s))\|^{\frac 12} \|\nabla u ;L^2_{\sharp}(\mathcal F(s))\| \|u ; H^2_{\sharp}(\mathcal F(s))\|^{\frac 12}{\mathrm d}s\, \, \\[8pt]
  \leq  C_{0}(t) \displaystyle\int_0^t   \|\partial_{txx} h ; L^2_{\sharp}(0,L)\|      \|\nabla u ;L^2_{\sharp}(\mathcal F(s))\| \|u ; H^2_{\sharp}(\mathcal F(s))\|^{\frac 12}\,.
\end{multline*}
Consequently,  we obtain:
\begin{multline*}
 \int_{\mathcal Q_t} |\nabla u| |u| |\nabla^2 \hat{\Lambda}|
 \leq C_0(t)\int_0^t \|\nabla u ; L^2_{\sharp}(\mathcal F(t))\|^4  \\+\varepsilon \int_0^t \left( \| \partial _t u  + u \cdot \nabla u  ; L^2_{\sharp}(\mathcal F(s))\|^2 + \|\partial_{txx} h ; L^2_{\sharp}(0,L)\|^2\right)\,.
\end{multline*}
Finally, for arbitrary small $\varepsilon >0,$ we have:
\begin{multline} \label{eq_estI5}
|I_5| \leq C_0(t) \int_0^t \left( \|\partial_{tx} h ; L^2_{\sharp}(0,L)\|^4 + \|\nabla u ; L^2_{\sharp}(\mathcal F(s))\|^4  \right){\mathrm d}s \\
+\varepsilon \int_0^t  \left( \| \partial _t u  + u \cdot \nabla u  ; L^2_{\sharp}(\mathcal F(s))\|^2 + \|\partial_{txx} h ; L^2_{\sharp}(0,L)\|^2\right){\mathrm d}s\,.
\end{multline}
Introducing \eqref{eq_estI1}-\eqref{eq_estI2}-\eqref{eq_estI3}-\eqref{eq_estI4}-\eqref{eq_estI5} into
\eqref{eq_estDubase}, we obtain that, for arbitrary $\varepsilon >0,$ 
there holds:
\begin{multline} \label{eq_estDu2}
\dfrac \mu2   \int_{\mathcal F (t)} |\nabla u|^2+ \dfrac 14 \int_{\mathcal Q_t}\rho_f |\partial_t u + u \cdot \nabla u|^2 + \int_0^t\int_{0}^{L} \phi(u,p,h) \partial_{tt} h 
 \\
\leq C_0(t) \int_0^t\left( \|\nabla u ; L^2_{\sharp}(\mathcal F(s))\|^4 +  \|\partial_{tx} h ; L^2_{\sharp}(0,L)\|^4  \right)
\\
+\varepsilon  \int_0^t\left( \| \partial _t u  + u \cdot \nabla u  ; L^2_{\sharp}(\mathcal F(s))\|^2 + \|\partial_{txx} h ; L^2_{\sharp}(0,L)\|^2\right)+\dfrac \mu2   \int_{\mathcal F (0)} |\nabla u_0|^2\,.
\end{multline} 

We now take care of the elastic part: we apply  \eqref{egalite_vitesse_beam} of Lemma \ref{lem:vitesse} for $b=h$ and rewrite \eqref{eq_mult_beam} as 
\begin{multline} \label{eq_lastidentity}
\int_0^t\int_{0}^{L} \rho_s|\partial_{tt} h|^2 + \int_{0}^{L} \left(\dfrac{\gamma}{2} |\partial_{tx} h(x, t)|^2 - \beta \partial_{t} h(x,t) \partial_{xx} h(x,t)  - \alpha \partial_{tx} h(x,t) \partial_{xxx} h (x,t)\right) \mathrm{d}x
\\
- \beta \int_0^t\int_{0}^{L} |\partial_{tx} h |^2 - \alpha\int_0^t \int_{0}^{L} |\partial_{txx} h |^2 =\int_0^t\int_{0}^{L} \phi(u,p,h) \partial_{tt} h + C_0
\end{multline}
We add \eqref{eq_estDu2} and \eqref{eq_lastidentity}, restricted to $\varepsilon < \min (\rho_f / 8,1)$ and bound all possible terms by energy estimate (Proposition \ref{prop_estkin}). This yields: 
\begin{multline} \label{eq_estDu22}
\dfrac \mu2   \int_{\mathcal F (t)} |\nabla u|^2+ \dfrac 18  \int_{\mathcal Q_t}\rho_f |\partial_t u + u \cdot \nabla u|^2 +\int_0^t\int_{0}^{L} \rho_s|\partial_{tt} h|^2  + Rem(t) \\  \leq C_0(t)\int_0^t\left( \|\nabla u ; L^2_{\sharp}(\mathcal F(s))\|^4 +  \|\partial_{tx} h ; L^2_{\sharp}(0,L)\|^4  \right)
+ (\alpha + 1 ) \int_0^t\int_0^L |\partial_{txx} h|^2 + C_0.
\end{multline} 
where, thanks to Proposition \ref{prop_estkin}:
$$
Rem(t) = \int_{0}^{L} \left(\dfrac{\gamma}{2} |\partial_{tx} h(x, t)|^2  
- \beta \partial_{t} h(x,t) \partial_{xx} h(x,t)  - \alpha \partial_{tx} h(x,t) \partial_{xxx} h (x,t)\right) \mathrm{d}x\,, \quad \forall \, t \in [0,T].
$$
satisfies:
$$
Rem(t) \geq  - C_0 - \dfrac{\alpha}{2} \left(  \|\partial_{xxx} h(\cdot,t) ; L^2(0,L)\|^2 + \|\partial_{tx} h(\cdot,t) ; L^2(0,L)\|^2\right)\,. 
$$
Finally, we get:
\begin{multline} \label{eq_estDu3}
\dfrac \mu2   \int_{\mathcal F (t)} |\nabla u|^2  - \dfrac{\alpha}{2} \int_0^L \left(|\partial_{xxx} h(x,t)|^2 + |\partial_{tx} h(x,t)|^2\right){\rm d}x \\
+ \dfrac 18  \int_{\mathcal Q_t}\rho_f |\partial_t u + u \cdot \nabla u|^2 +\int_0^t\int_{0}^{L} \rho_s|\partial_{tt} h|^2\\  \leq C_0(t)\int_0^t\left( \|\nabla u ; L^2_{\sharp}(\mathcal F(s))\|^4 +  \|\partial_{tx} h ; L^2_{\sharp}(0,L)\|^4  \right)
+ (\alpha + 1 ) \int_0^t\int_0^L |\partial_{txx} h|^2 + C_0.
\end{multline} 

To go further, we thus need to obtain an estimate for 
$$
\int_0^t\int_0^L |\partial_{txx} h|^2 \,, \qquad \int_0^L \left(|\partial_{xxx} h(x,t)|^2 + |\partial_{tx} h(x,t)|^2\right){\rm d}x
$$ 
 The next step is thus an independant regularity estimate for the beam equation.

\medskip

\paragraph{\em Step 2: Regularity estimate for structure equation}
First, $- \partial_{txx} h \in L^2_{\sharp}((0,L) \times (0,T))$ so  it is a suitable multiplier for \eqref{eq_elasticite}.
We obtain after integration:
\begin{multline*}
-\int_0^t \int_{0}^{L} \rho_s\partial_{tt} h \partial_{txx} h + \beta \int_0^t \int_{0}^{L} \partial_{xx} h \partial_{txx} h - \alpha \int_0^t \int_{0}^{L} \partial_{xxxx} h \partial_{txx} h
+ \gamma \int_0^t \int_{0}^{L} |\partial_{txx} h|^2\\
= - \int_0^t \int_{0}^{L} \phi(u,p,h) \partial_{txx} h\,.
\end{multline*}
We can further integrate by parts
\begin{eqnarray} \label{eq_estdtxhv1}
\dfrac{1}{2} \int_{0}^{L} \left[\rho_s|\partial_{tx} h(x,t)|^2 + \alpha |\partial_{xxx} h(x,t)|^2 + \beta |\partial_{xx} h(x,t)|^2 \right] {\mathrm d}x + \gamma\int_0^t \int_{0}^{L} |\partial_{txx} h|^2 =  \\ 
 \nonumber
-\int_0^t \int_{0}^{L} \phi(u,p,h) \partial_{txx} h{\rm d}x+ \dfrac{1}{2} \int_{0}^{L} \left[\rho_s|\partial_{x} \dot h_0(x)|^2 + \alpha |\partial_{xxx} h_0(x)|^2 + \beta |\partial_{xx} h_0(x)|^2 \right] {\mathrm d}x.
\end{eqnarray}
Note that all terms make sense due to the regularity of the strong solution $h$.
We control the first term of the right-hand side with the help of {Corollary \ref{cor_phi}}, recalling that constant $K^b(R_t)$ defines a function of time that remains bounded on bounded subintervals of $[0,T^*)$. Consequently, we obtain:
\begin{eqnarray}\notag
\left|  \int_{0}^{L} \phi(u,p,h) \partial_{txx} h \right| &\leq& \|\phi(u,p,h) ; H^{\frac 12}_{\sharp}(0,L)\| \|\partial_{txx} h ; H^{-\frac 12}_{\sharp}(0,L)\|\, \\\notag
&\leq& C_{0}(t) \left( \|\mu {\Delta} u - \nabla p ; L^2_{\sharp}(\mathcal F(t))\|  +\|\partial_t h ; H^{\frac 32}_{\sharp}(0,L)\|  \right) \|\partial_{t} h ; H^{\frac 32}_{\sharp}(0,L) \|\, .
\end{eqnarray}
Since $\mu {\Delta} u - \nabla p= \partial_{t} u + u \cdot \nabla u$ we have then
\begin{eqnarray}\notag
\left|  \int_{0}^{L} \phi(u,p,h) \partial_{txx} h \right| \leq C_{0}(t) \left( \|\partial_t u+u\cdot \nabla u ; L^2_{\sharp}(\mathcal F(t))\|  +\|\partial_t h ; H^{\frac 32}_{\sharp}(0,L)\|  \right) \|\partial_{t} h ; H^{\frac 32}_{\sharp}(0,L) \|\, .
\end{eqnarray}
Moreover, we apply the following interpolation inequality 
\begin{eqnarray*}
\|\partial_{t} h ; H^{\frac 32}_{\sharp}(0,L)\| &\leq&  C \|\partial_{tx} h ; L^2_{\sharp}(0,L)\|^{\frac 12} \|\partial_{txx} h ; L^2_{\sharp}(0,L)\|^{\frac 12}\,.
\end{eqnarray*}
this yields:
\begin{multline} \label{eq_estphih}
\left|  \int_{0}^{L} \phi(u,p,h) \partial_{txx} h \right| \\
\leq
C_0(t)  \|\partial_{tx} h ; L^2_{\sharp}(0,L)\|^2  + \varepsilon \left(  \|\partial_t u+u\cdot \nabla u ; L^2_{\sharp}(\mathcal F(t))\|^2 + \|\partial_{txx} h ; L^2_{\sharp}(0,L)\|^2 \right)\,.  
\end{multline}
Introducing \eqref{eq_estphih} into \eqref{eq_estdtxhv1}, we obtain that, for arbitrary $\varepsilon>0$
sufficiently small, there holds:
\begin{multline} \label{eq_estdtxh}
\dfrac{1}{2} \int_{0}^{L} \left[ \rho_s|\partial_{tx} h(x,t)|^2 + \alpha|\partial_{xxx} h(x,t)|^2 \right] {\mathrm d}x+ \dfrac{\gamma}{2} \int_0^t\int_{0}^{L} |\partial_{txx} h|^2  \\
\leq 
C_0(t) \int_0^t\|\partial_{tx} h ; L^2_{\sharp}(0,L)\|^2  + \varepsilon \int_0^t \|\partial_t u+u\cdot \nabla u ; L^2_{\sharp}(\mathcal F(s))\|^2 + C_0
 \end{multline}

\paragraph{\em Step 3 : Conclusion}
We multiply \eqref{eq_estdtxh} by a sufficently large constant, typically:
$$
\max\left(\dfrac \alpha\rho_s, 4, \dfrac{4(\alpha+1)}{\gamma}\right)
$$ 
and add \eqref{eq_estDu3}, we obtain, choosing $\varepsilon$ sufficiently small:
\begin{multline} \label{eq_Dufinal}
\dfrac \mu2 \int_{\mathcal F (t)} |\nabla u|^2 
 + 
   \int_{\mathcal Q_t} |\partial_t u + u \cdot \nabla u|^2\\ +\delta\left(\dfrac{1}{2} \int_{0}^{L} \left[ \rho_s|\partial_{tx} h(x,t)|^2 + \alpha|\partial_{xxx} h(x,t)|^2 \right] {\mathrm d}x+ \dfrac{\gamma}{2} \int_0^t\int_{0}^{L} |\partial_{txx} h|^2\right)
 \\\leq
C_{0}(t)
\int_0^t\left( 1 +  \|\nabla u ; L^2(\mathcal F(s))\|^2 +  \|\partial_{tx} h ; L^2_{\sharp}(0,L)\|^2  \right) ^2  +C_0  \,,
\end{multline} 
where $\delta >0$. Setting 
$$
\mathcal E_{reg} := \dfrac \mu2 \int_{\mathcal F (t)} |\nabla u|^2 
 +\delta\left(\dfrac{1}{2} \int_{0}^{L} \left[ \rho_s|\partial_{tx} h(x,t)|^2 + \alpha|\partial_{xxx} h(x,t)|^2 \right] {\mathrm d}x\right).
$$
we have then for all $t \in [0,T]$:
$$
\|u(\cdot,t) ; H^1(\mathcal F(t))\|^2  + \|\partial_t h(\cdot,t) ; H^1(0,L)\|^2 + \|h ; H^3(0,L)\| \leq C_0 + C \mathcal E_{reg}(t).
$$ 
with a constant $C$ depending only on $\mu,\rho_s,\alpha$ and $\delta,$ and
$$
 \mathcal E_{reg}(t) \leq C_{0}(t)\int_0^t\left( 1 +  \|\nabla u ; L^2(\mathcal F(s))\|^2 +  \|\partial_{tx} h ; L^2_{\sharp}(0,L)\|^2  \right)  \left(\mathcal E_{reg}(s) +1\right){\mathrm d}s +C_0
$$
Proposition \ref{prop_estkin} implies then that:
$$
 \int_{0}^t \left( 1 +  \|\nabla u ; L^2(\mathcal F(t))\|^2 +  \|\partial_{tx} h ; L^2_{\sharp}(0,L)\|^2  \right)  \leq C_{0}(t)
$$
We thus complete the proof  of Proposition \ref{prop_estreg} by applying a Gronwall lemma and remembering that $t\mapsto C_0(t)$ is bounded on any bounded interval.

\appendix 

\section{Technical details for  the no-contact result} 
In this appendix we collect  technical lemmas that are used throughout Section \ref{sec_glob}.

\subsection{Estimating positivity of scalar functions} \label{app_inequalities}
We prove first functional inequalities which enable to bound from
below a positive function. 
\begin{proposition} \label{prop_below1}
Let $\alpha \in (0,1/2) \setminus \{1/4\}.$ 
Given a non-negative function $\eta \in H^2_{\sharp}(0,L),$  there holds:
$$
\int_{0}^{L} \dfrac{|\partial_x \eta(x)|^4}{|\eta(x)|^{4\alpha}} {\rm d}x  \leq C \|\partial_{xx} \eta ; L^2_{\sharp}(0,L)\|^2 \|\eta ; L^{\infty}_{\sharp}(0,L)\|^{2(1-2\alpha)}\,.
$$ 
\end{proposition}
\begin{proof}
Without restriction, we assume that $\eta$ is smooth and we set 
$$ 
b(x) = [\eta(x)]^{1-\alpha} \quad \forall \, x \in (0,L)\,.
$$
A straightforward integration by parts together with the L-periodicity, yields 
$$
\int_{0}^{L} |\partial_x b(x)|^4 {\rm d}x  + 3 \int_{0}^{L} b(x) \partial_{xx} b(x) |\partial_x b (x)|^2 {\rm d}x= 0\,.
$$
Replacing $b$ by its value, we obtain:
$$
(1-\alpha)^3 (  1- 4 \alpha) \int_{0}^{L} \dfrac{|\partial_x \eta(x)|^4}{|\eta(x)|^{4\alpha}}{\rm d}x + 3 (1-\alpha)^3 \int_{0}^{L} \left[ |\eta(x)|^{(1- 2\alpha) } \partial_{xx} \eta(x)\right] \dfrac{|\partial_x \eta(x)|^2}{|\eta(x)|^{2\alpha}} {\rm d}x = 0\,.
$$
When $\alpha \neq 1/4$ this yields 
$$
 \int_{0}^{L} \dfrac{|\partial_x \eta(x)|^4}{|\eta(x)|^{4\alpha}}{\rm d}x   \leq C \int_{0}^{L} \left[ |\eta(x)|^{(1- 2\alpha) } \partial_{xx} \eta(x)\right] \dfrac{|\partial_x \eta(x)|^2}{|\eta(x)|^{2\alpha}} {\rm d}x.
$$
We conclude by applying Cauchy-Schwarz inequality, which yields, since $\alpha < 1/2$
\begin{eqnarray*}
 \int_{0}^{L} \dfrac{|\partial_x \eta(x)|^4}{|\eta(x)|^{4\alpha}}{\rm d}x  &\leq& C \int_{0}^{L}  |\eta(x)|^{2(1- 2\alpha) } |\partial_{xx} \eta(x)|^2{\rm d}x\\
 													& \leq  &C \|\eta ; L^{\infty}_{\sharp}(0, L)\|^{2(1-2\alpha)} \|\partial_{xx} \eta ; L^2_{\sharp}(0,L)\|^2\,.
\end{eqnarray*}

\end{proof}

\begin{proposition} \label{prop_below2}
Given a non-negative function $\eta \in H^3_{\sharp}(0,L),$  the following pointwise estimates hold true:
\begin{eqnarray}
\label{esteta1}|\partial_x \eta(x)|^2  \leq C  \|\partial_{xx}\eta ; L^2_{\sharp}(0,L)\|^{\frac{1}{2}} \|\partial_{xxx} \eta ; L^2_{\sharp}(0,L)\|^{\frac{1}{2}}   \, |\eta(x)|\,, && \forall \, x \in (0,L)\,.\\
\label{esteta2}|\partial_x \eta(x)|^2  \leq C \|\partial_{xxx} \eta ; L^2_{\sharp}(0,L)\| \, |\eta(x)|\,, && \forall \, x \in (0,L)\,.
\end{eqnarray}
\end{proposition}
\begin{proof}
Let us denote $z\in H^3_{\sharp}(0,L)$  a non-negative function. 
\begin{equation} \label{prmax_z}
|\partial_x z(x)| \leq \||\partial_x z|^2 + z\partial_{xx} z ; L^{\infty}_{\sharp}(0,L)\|^{\frac{1}{2}}\,, \quad \forall \, x \in (0,L).
\end{equation}
Indeed, as $z$ is $L$-periodic, $\partial_x z$ reaches its maximal and minimal values. 
Then, if $\partial_x z$ is maximal at $x_0 \in (0,L)$, then  $\partial_{xx} z(x_0) = 0$ and consequently
$$
|\partial_x z(x_0)| \leq \sqrt{|\partial_x z(x_0)|^2 + z\partial_{xx} z(x_0)} \leq  \||\partial_x z|^2 + z\partial_{xx} z ; L^{\infty}_{\sharp}(0,L)\|^{\frac{1}{2}}\,.
$$
We get a similar inequality where $\partial_x z$ is minimal and obtain \eqref{prmax_z}. 
 
 \medskip
 
We apply now the previous estimate with
$$
z(x) = \sqrt{\eta(x)} \,, \quad \forall \, x \in  (0,L).
$$
Of course $z$ is a non-negative $L$-periodic function which belongs to $H^3_\sharp(0,L)$ so  it satisfies \eqref{prmax_z}.
Replacing $z$ by $\sqrt{\eta}$  leads to
$$
|\partial_x \eta(x)|  \leq C \sqrt{\|\partial_{xx} \eta ; L^{\infty}_{\sharp}(0,L)\|\, \eta(x)} \,, \quad \forall \, x \in (0,L).$$ 
Thanks to the continuous embedding $H^1_{\sharp}(0,L) \cap L^{2}_{\sharp,0}(0,L)\,$ in $L^{\infty}_{\sharp}(0,L)$, we obtain \eqref{esteta2} and the interpolation inequality $\|\partial_{xx} \eta ; L^{\infty}_{\sharp}(0,L)\|\leq \|\partial_{xx} \eta ; L^{2}_{\sharp}(0,L)\|^{1/2}\|\partial_{xxx} \eta ; L^{2}_{\sharp}(0,L)\|^{1/2}$ implies that \eqref{esteta1} is satisfied.

\end{proof}

\begin{proposition} \label{prop_controlparH2}
There exists a continuous function $D_{min} : \mathbb (0,\infty) \times \mathbb R^+ \to (0,\infty)$
such that, given a non-negative function $\eta \in H^2_{\sharp}(0,L)$ there holds 
\begin{equation} \label{eq_controlparH2}
\|\eta^{-1}  \, ; \,  L^{\infty}_{\sharp}(0,L)\| \leq D_{min} \left( \|\eta^{-1} \, ; \, L^1_{\sharp} (0,L)\|, \|\eta, H^2_{\sharp}(0,L)\| \right)\,.
\end{equation}
\end{proposition} 
\begin{proof}
Without further restriction, we assume $\eta$ to be smooth. We note that $\|\eta^{-1}  \, ; \,  L^{\infty}_{\sharp}(0,L)\|$ is achieved  for some
$x_0 \in [0,L].$ Then, as $\partial_x \eta(x_0) =0$ we have:
$$
\eta(x) = \eta(x_0) + \int_{x_0}^{x} (s-x_0)\partial_{xx} \eta(s) {\rm d}s\,,
$$
At this point, we consider two cases. 
First, if $\|\partial_{xx} \eta ; H^2_{\sharp}(0,L)\| =0,$ we obtain
\begin{equation} \label{eq_case1}
\|\eta^{-1} ; L^1_{\sharp}(0,L)\| = L\|\eta^{-1} ; L^{\infty}_{\sharp}(0,L)\|.
\end{equation}
Second, if  $\|\partial_{xx} \eta ; H^2_{\sharp}(0,L)\| > 0,$
we have the bound:
$$
\eta(x) \leq \eta(x_0) +\dfrac{1}{\sqrt{3}} \|\eta; H^2_{\sharp}\| |x-x_0|^{\frac{3}{2}}\,, \quad \forall \, x \in (x_0 - L, x_0+L).
$$
For simplicity, let denote $\eta_m := \eta(x_0)$ and $M = \|\eta; H^2_{\sharp}(0,L)\|/3.$ We have
then 
\begin{eqnarray*}
\|\eta^{-1};L^1_{\sharp}(0,L)\|& =&  \int_{x_0}^{x_0+L} \dfrac{{\rm d}s}{\eta(s)} 
			\geq  \int_{0}^{L} \dfrac{{\rm d}s}{ \eta_{m} + M |s|^{\frac{3}{2}}} 
			 \geq \dfrac{1}{M^{\frac{2}{3}}\eta_m^{\frac{1}{3}}} \phi\left( \left[\dfrac{M}{\eta_m}\right]^{\frac{2}{3}}\right)\,,
\end{eqnarray*}
where we applied a straightforward change of variable to introduce 
$$
\phi(\sigma) := \int_{0}^{L \sigma} \dfrac{{\rm d}z}{(1+|z|^{\frac{3}{2}})}\,. 
$$
We remark that  $\phi : \mathbb R^+ \to \mathbb R$ is an increasing continuous function such that
$$
\phi(\sigma) \sim 
\left\{
\begin{array}{ll}
L\sigma & \text{ for $\sigma << 1$}\,, \\[8pt]
\displaystyle\int_{\mathbb R} \dfrac{{\rm d}z}{{1+|z|^{\frac{3}{2}}}} & \text{ for $\sigma>>1.$}
\end{array}
\right.
$$
Applying furthermore that $\|\eta^{-1} ; L^1_{\sharp}(0,L)\| \leq  L \|\eta^{-1} ; L^{\infty}_{\sharp}(0,L)\| = L \eta_m$, 
we obtain
$$
\|\eta^{-1} ; L^{1}_{\sharp}(0,L)\| \geq \dfrac{1}{M^{\frac 23}\eta_m^{\frac 13}} \phi\left( \left[\frac{{M} \|\eta^{-1} ; L^{1}_{\sharp}(0,L)\| }{L}\right]^{\frac{2}{3}}\right)
$$
Finally, if $\|\eta ; H^2_{\sharp}(0,L)\| >0,$ we have :
\begin{equation} \label{eq_case2}
\|\eta^{-1} ; L^{\infty}_{\sharp}(0,L)\| \leq  \|\eta ; H^2_{\sharp}(0,L)\|^{2}\|\eta^{-1} ; L^{1}_{\sharp}(0,L)\|^3  {\left[ \phi\left( \left[\dfrac{ \|\eta ; H^2_{\sharp}(0,L)\| \|\eta^{-1} ; L^{1}_{\sharp}(0,L)\|}{L} \right]^{\frac{2}{3}}\right) \right]^{-3}},
\end{equation}
The above expansion of $\phi(\sigma)$ for small values of $\sigma$ yields that, when $\|\eta ; H^2_{\sharp}(0,L)\| << 1$ with $\|\eta^{-1} ; L^{1}_{\sharp}(0,L)\|$ bounded, there holds 
$$
  \|\eta ; H^2_{\sharp}(0,L)\|^{2}\|\eta^{-1} ; L^{1}_{\sharp}(0,L)\|^3 { 
 \left[ 
 \phi\left( \left[\dfrac{ \|\eta ; H^2_{\sharp}(0,L)\| \|\eta^{-1} ; L^{1}_{\sharp}(0,L)\|}{L} \right]^{\frac{2}{3}}\right)
 \right]^{-3}
 } 
  \sim  \dfrac{1}{L} \|\eta^{-1} ; L^1_{\sharp}(0,L)\|\,.
$$
Hence, we set 
$$
D_{min} (\alpha, \beta) = 
\dfrac{\alpha}{L}\,, \text{ if $\beta = 0,$}
\qquad
D_{min} (\alpha, \beta) = 
\dfrac{\beta^{2}\alpha^3}{9} {\left[\phi\left( \left[\dfrac{ \beta\alpha}{L} \right]^{\frac{2}{3}}\right)\right]^{-3}}\,,  \text{ else}\,.
$$
Because of the previous arguments, this is a continuous function on $(0,\infty) \times [0,\infty)$
which satisfies \eqref{eq_controlparH2}.
\end{proof}

\subsection{Estimates on the stream-function $\psi$} \label{app_psi}
In this appendix, we gather technical estimates regarding the stream-function
$$
\psi(x,y,t) = \partial_{x} h(x,t) \chi_0\left( \dfrac{y}{h(x,t)}\right) \quad \forall \, (x,y,t ) \in \mathcal Q_T,
$$
where $h \in H^{2}(0,T; L^2_{\sharp}(0,L)) \cap L^2(0,T;H^4_{\sharp}(0,L))$ is given and remains strictly positive on $(0,T).$
We recall that 
$$
\chi_0(z) = z^2(3-2z) \quad \forall \, z \in (0,1)\,,
$$
that  $\mathcal{F}(t)$ stands for the fluid domain at time $t$ and $\mathcal Q_T$ is the time-space
fluid domain.  With these notations, we prove:
\begin{proposition}
\label{propo:psi}
There exists a universal constant $C<\infty$ for which:
\begin{itemize}
\item for all $t \in (0,T)$ 
\begin{eqnarray} \label{eq_nablapsi}
& |\nabla \psi(x,y,t)| \leq C\Bigg[ |\partial_{xx} h(x,t)| + \dfrac{|\partial_{x} h(x,t)|}{h(x,t)}  +  \dfrac{|\partial_{x} h(x,t)|^2}{h(x,t)}\Bigg] \,, \quad \forall \, (x,y) \in \mathcal F(t)\,,&
\\\label{eq_dxpsi}
&\|\partial_y \psi (\cdot,t) ; L^2(\mathcal F (t))\| \leq C   \|h ; L^{\infty}_{\sharp}(0,L)\|^{\frac{1}{4}} \|\partial_{xx} h ; L^2_{\sharp}(0,L)\|^{\frac{1}{2}} \left[ \displaystyle{\int_{0}^{L}} \dfrac {{\rm d}x}{h(x,t)}   \right]^{\frac{1}{4}}\,,& \\[8pt]\label{eq_dypsi}
&\|\partial_x \psi (\cdot,t) ; L^2(\mathcal F (t))\| \leq C \|h ; L^{\infty}_{\sharp}(0,L)\|^{\frac{1}{2}}\|\partial_{xx} h ; L^2_{\sharp}(0,L)\|\,;&
\end{eqnarray}
\item on the whole time-space domain 
\begin{eqnarray}\label{eq_dtpsi}
\|\partial_{t} \psi ; L^2(\mathcal Q _T)\| &\leq& C\Bigg[ \int_0^T  \Big( \|h ; L^{\infty}_{\sharp}(0,L)\| \|\partial_{xt} h ; L^2_{\sharp}(0,L)\|^2\\&&\qquad\nonumber + \|\partial_t h ;L^2_{\sharp}(0,L)\|^2 \|\partial_{xxx} h ;L^2_{\sharp}(0,L)\|\Big) \Bigg]^{\frac{1}{2}}\,,\\ \label{eq_dxxpsi}
\|\partial_{xx} \psi ; L^2(\mathcal Q _T)\| &\leq& C \Bigg[\int_{0}^T \Big(  \|h ; L^{\infty}_{\sharp}(0,L)\| \|\partial_{xxx} h ; L^2_{\sharp}(0,L)\|^2 \\&&\qquad\nonumber +  \|\partial_{xx} h ; L^2_{\sharp}(0,L)\|^{\frac{3}{2}}\|\partial_{xxx}h ; L^2_{\sharp}(0,L)\|^{\frac{3}{2}}\Big) \Bigg]^{\frac 12}\,. 
\end{eqnarray}
\end{itemize}

\end{proposition}
\begin{proof}
We start by the time-dependant estimate. We compute the partial derivative of the stream-function with respect to the space variables:
$$
\left\{
\begin{array}{rcl}
\partial_x \psi(x,y,t)& =& \partial_{xx} h(x,t)  \chi_0\left( \dfrac{y}{h(x,t)}\right)  - \dfrac{|\partial_x h(x,t)|^2y}{|h(x,t)|^2}  \chi'_0\left( \dfrac{y}{h(x,t)}\right)\,,\\[6pt]
\partial_y \psi(x,y,t) &= &\dfrac{\partial_x h(x,t)}{h(x,t)}  \chi'_0\left( \dfrac{y}{h(x,t)}\right)\,,
\end{array}
\right.
 \quad \forall \, (x,y,t) \in \mathcal Q_T.
$$
Thus \eqref{eq_nablapsi} is satisfied. Furthermore, we have, for all $t \in (0,T)$, recalling that $\chi_0 \in C^{\infty}([0,1])$: 
\begin{eqnarray*}
\|\partial_{y} \psi ; L^2(\mathcal F(t))\|^2 &=& \int_{0}^{L}\int_0^{h(x,t)} \dfrac{|\partial_x h(x,t)|^2}{|h(x,t)|^2} \left| \chi'_0\left( \dfrac{y}{h(x,t)}\right)\right|^2 {\rm d}y {\rm d}x \\
&=&   \int_{0}^{L}\int_0^1 \dfrac{|\partial_x h(x,t)|^2}{|h(x,t)|} \left| \chi'_0(z) \right|^2 {\rm d}y {\rm d}z. \\
\end{eqnarray*}


Then applying  Proposition \ref{prop_below1} with $\alpha=3/8$, we have:
\begin{eqnarray*}
\|\partial_{y} \psi ; L^2(\mathcal F(t))\|^2 
& \leq C & \left[ \int_{0}^{L} \dfrac{|\partial_x h(x,t)|^4}{|h(x,t)|} {\rm d}x\right]^{\frac{1}{2}} \left[ \int_{0}^{L} \dfrac {{\rm d}x}{h(x,t)}   \right]^{\frac{1}{2}}\,\\
& \leq C & \|h ; L^{\infty}_{\sharp}(0,L)\|^{\frac{1}{4}}\left[ \int_{0}^{L} \dfrac{|\partial_x h(x,t)|^4}{|h(x,t)|^{\frac32}} {\rm d}x\right]^{\frac{1}{2}} \left[ \int_{0}^{L} \dfrac {{\rm d}x}{h(x,t)}   \right]^{\frac{1}{2}}\\
& \leq C & \|h ; L^{\infty}_{\sharp}(0,L)\|^{\frac{1}{2}} \|\partial_{xx} h ; L^2_{\sharp}(0,L)\| \left[ \int_{0}^{L} \dfrac {{\rm d}x}{h(x,t)}   \right]^{\frac{1}{2}}.
\end{eqnarray*}
With similar arguments, we easily obtain
$$
\|\partial_{x} \psi ; L^2(\mathcal F(t))\|^2   \leq C \left[ \|h ; L^{\infty}_{\sharp}(0,L)\|\|\partial_{xx} h ; L^2_{\sharp}(0,L)\|^2 + \int_{0}^{L} \dfrac{|\partial_x h(x,t)|^4}{h(x,t)} {\rm d}x\right].
$$
Applying again Proposition \ref{prop_below1} with $\alpha=3/8$, the following estimate holds true in a similar  way as above
$$
\|\partial_{x} \psi ; L^2(\mathcal F(t))\|^2   \leq C  \|h ; L^{\infty}_{\sharp}(0,L)\|\|\partial_{xx} h ; L^2_{\sharp}(0,L)\|^2 \,.
$$
This ends the proof of \eqref{eq_dxpsi}-\eqref{eq_dypsi}.
\medskip

We  now take care of the two other estimates. A simple calculation gives
$$
\partial_t \psi(x,y,t) =   \partial_{xt} h(x,t)  \chi_0\left( \dfrac{y}{h(x,t)}\right)  - \dfrac{ \partial_x h(x,t) \partial_t h(x,t) y}{|h(x,t)|^2}  \chi'_0\left( \dfrac{y}{h(x,t)}\right)\,,\quad 
\forall \, (x,y,t) \in \mathcal Q_T.
$$
Consequently, we have 
$$
\|\partial_{t} \psi ; L^2(\mathcal Q_T)\|^2   \leq C \int_0^T \left( \|h ; L^{\infty}_{\sharp}(0,L)\| \|\partial_{xt} h ; L^2_{\sharp}(0,L)\|^2 + \int_{0}^{L} \dfrac{ |\partial_x h(x,t)|^2 |\partial_t h(x,t)|^2}{h(x,t)} {\rm d}x \right)\,.
$$
With Proposition \ref{prop_below2} we bound the ratio ${|\partial_x h(x,t)|^2}/{h(t, x)}$ in the right-hand side, and we obtain
$$
\|\partial_{t} \psi ; L^2(\mathcal Q_T)\|^2  \leq C \int_0^T  \left( \|h ; L^{\infty}_{\sharp}(0,L)\| \|\partial_{xt} h ; L^2_{\sharp}(0,L)\|^2 + \|\partial_t h ;L^2_{\sharp}(0,L)\|^2 \|\partial_{xxx} h ;L^2_{\sharp}(0,L)\|\right)\,.
$$
Finally, we estimate $\partial_{xx} \psi$ which is equal to
$$
\partial_{xx} \psi (x,y, t)= \partial_{xxx} h(x,t) \, \chi_0 (\frac{y}{h(x,t)} )- \dfrac{3 y \partial_{xx} h(x,t)  \partial_x h(x,t) }{(h(x,t))^2}\chi_0' (\frac{y}{h(x,t)} )+ \dfrac{(\partial_x h(x,t))^3 y^2}{(h(x,t))^4}\chi_0^{''}(\frac{y}{h(x,t)} )\,,
$$
Consequently, the following estimate holds true
\begin{eqnarray*}
\|\partial_{xx} \psi ; L^2(\mathcal Q_T)\|^2  &\leq &C \int_{0}^T \left[ \phantom{\int} \|h ; L^{\infty}_{\sharp}(0,L)\| \|\partial_{xxx} h ; L^2_{\sharp}(0,L)\|^2\right. \\
&&\quad+\left.\int_{0}^{L} \left(\dfrac{ |\partial_{xx} h(x,t)|^2 |\partial_x h(x,t)|^2}{h(x,t)} + \dfrac{|\partial_{x} h(x,t)|^6}{|h(x,t)|^3}{\rm d}x \right)\right]\,.
\end{eqnarray*}
By estimates \eqref{esteta1} and \eqref{esteta2}  of Proposition \ref{prop_below2}, we obtain 
\begin{eqnarray*}
\|\partial_{xx} \psi ; L^2(\mathcal Q_T)\|^2 &\leq& C\int_{0}^T \Bigl[  \|h ; L^{\infty}_{\sharp}(0,L)\| \|\partial_{xxx} h ; L^2_{\sharp}(0,L)\|^2 \\
&&\quad + \|\partial_{xx} h ;L^2_{\sharp}(0,L)\|^2 \|\partial_{xxx} h; L^2_{\sharp}(0,L)\| +  \|\partial_{xx} h ; L^2_{\sharp}(0,L)\|^{\frac{3}{2}}\|\partial_{xxx}h ; L^2_{\sharp}(0,L)\|^{\frac{3}{2}}\Bigr]\,.
\end{eqnarray*}
Finally \eqref{eq_dxxpsi} is satisfied. 
\end{proof}

\medskip

\section{Technical lemmas for the regularity estimates}\label{App:reg}
In this second Appendix, we collect technical results that are applied in the proof of {Proposition \ref{prop_estreg}}.
We first construct the lifting velocity-field $\hat{\Lambda}$ and prove then two identities that are central in the proof of our regularity
estimates.

\subsection{Construction of $\hat \Lambda$}
We construct a time-frozen operator $U_h$ which satisfies the equivalent properties to the one we require for $\hat{\Lambda}.$ We shall then set $\hat{\Lambda} = U_{h}[\partial_t h].$ The construction 
of $U_h$ is the content of the following proposition:  
\begin{proposition} \label{prop_uetoile}
Let  $h \in H^2_{\sharp}(0,L)$ such that $h^{-1} \in L^{\infty}_{\sharp}(0,L),$ there exists a continuous linear mapping  $U_{h}  : H^{ 1}_{\sharp}(0,L) \cap L^{2}_{\sharp,0}(0,L)\to  H^1_{\sharp}(\Omega_h)$ s.t. :
\begin{itemize}
\item for all $\dot{\eta} \in H^1_{\sharp}(0,L) \cap  L^{2}_{\sharp,0}(0,L)$ we have 
\begin{align*}
& U_h[\dot{\eta}](x,0) = 0\,, \quad U_h[\dot{\eta}](x,h(x)) = \dot{\eta}(x) e_2 \,, \quad \forall \, x \in (0,L)\,,\\
& {\rm div} \, U_h[\dot{\eta}] = 0 \qquad \text{on $\Omega_h$}
\end{align*}
\item for all $\dot{\eta} \in H^2_{\sharp}(0,L) \cap  L^{2}_{\sharp,0}(0,L) $ there holds $U_h[\dot{\eta}] \in H^2_{\sharp}(\Omega_h)\,.$
\end{itemize}
Furthermore we construct $U_h$ such that: 
\begin{itemize}
\item 
$
\partial_2 U_h[\dot{\eta}](x,h(x)) = 0\,.
$
\item there exists a constant $K^l$ depending increasingly on $\|h ; H^2_{\sharp}(0,L)\| + \|h^{-1} ; L^{\infty}_{\sharp}(0,L)\|$ for which
\begin{eqnarray}
\label{est:Uh:L2}\|U_h[\dot{\eta}] ; L^2_{\sharp}(\Omega_h)\| &\leq& K^l \|\dot{\eta} ; L^2(0,L)\|\,,\\
\label{est:Uh:H1}\|U_h[\dot{\eta}] ; H^1_{\sharp}(\Omega_h)\| &\leq& K^l \|\dot{\eta} ; H^1(0,L)\|\,, \\
\label{est:Uh:H2}\|U_h[\dot{\eta}] ; H^2_{\sharp}(\Omega_h)\| &\leq& K^l \|\dot{\eta} ; H^2(0,L)\|\,, \quad \forall \, \dot{\eta} \in H^{2}_{\sharp}(0,L) \cap L^2_{\sharp,0}(0,L)\,.
\end{eqnarray} 
\end{itemize}
\end{proposition}

\begin{proof}

We consider $h \in H^2_{\sharp}(0,L)$ which does not vanish on $(0, L)$ and assume that 
$$
\|h ; H^2_{\sharp}(0,L)\| + \|h^{-1} ; L^{\infty}_{\sharp}(0,L)\| \leq R_0.
$$
We construct $U_h$ in order that \eqref{est:Uh:L2}-\eqref{est:Uh:H1}-\eqref{est:Uh:H2} holds with a constant $K^l$ depending on $R_0$ only.  

\medskip

Fix $\lambda = 1/(2R_0)$ so that $\min_{x\in [0,L]} h(x) \geq 2 \lambda$. Assume that $\dot{\eta} \in H^1_{\sharp}(0, L).$
We  define $U_h$ as
\begin{equation}
\label{def:Uh}
U_h[\dot\eta]=\left\{
\begin{array}{ll}
\dot \eta e_2, \mbox{ in } \Omega_h\setminus \overline{ (0,L)\times (0, \lambda)},\\[2ex]
U_\lambda[\dot\eta] , \mbox{ in } (0,L)\times (0, \lambda).
\end{array}
\right.
\end{equation}
Note that the restriction of $U_h$ to $\Omega_h\setminus \overline{ (0,L)\times (0, \lambda)}$ obviously satisfies continuity estimates similar to   \eqref{est:Uh:L2}-\eqref{est:Uh:H1}-\eqref{est:Uh:H2}.
With this step, we reduce the construction to a rectangular box. In order to change the boundary on which we have to match the data, we define $U_\lambda[\dot{\eta}]$ by:
$$
U_\lambda[\dot{\eta}](x,y) = \begin{pmatrix} -\frac 1\lambda  & 0 \\ 0 & 1\end{pmatrix}\tilde{U}[\dot{\eta}]\left(x, \dfrac{\lambda - y}{\lambda} \right)\,, \quad \forall \, (x,y) \in  (0,L)\times (0, \lambda)\,.
$$

Let us define now $\tilde{U}[\dot{\eta}]$. We expand $\dot{\eta}\in H^2_\sharp(0, L)$ in Fourier
series as
$$
\dot{\eta}(x) = \sum_{n \in \mathbb Z} \dot{\eta}_{n} \exp (in  {\pi x\over L})\,.
$$ 
By assumption, we have $\dot{\eta}_0 =0\,.$ Let $k \in \mathbb N.$ Then, choosing $Q$ a polynomial function such that:
\begin{eqnarray}
Q(0) &=& 1  \quad Q'(0) = Q^{''}(0) = 0 \, \\
Q^{(l)}(1) &=& 0 \quad \forall \, l \leq k+1\,.   
\end{eqnarray}
and, for all $n\in \mathbb Z$: 
$$ 
P_n (z) := Q(\min(|n|z,1)) \quad \forall \, z \in [0,\infty)\,,
$$
we obtain that $P_n \in C^{k+1}([0,1])\,.$ Consequently, we define 
$$
\tilde{U} [\dot{\eta}]= \nabla^{\bot} \Psi, \quad \mbox{ with }\Psi(x,z) := \sum_{n\in \mathbb Z} \dfrac{L\dot{\eta}_n}{in\pi} P_{n}(z) \exp(in{\pi x\over L})\,.
$$
With this definition we have 
\begin{eqnarray}
\label{eq_Ut1}&& \text{$\tilde{U}[\dot{\eta}](x,0) = \dot{\eta}(x)e_2 $ and $U[\dot{\eta}](x,1) = 0$ for all $x \in (0,L)$,}\\[4pt]
\label{eq_Ut2}&& \partial_2 \tilde{U}[\dot{\eta}](x,0) = 0\,, \\[4pt]
\label{eq_Ut3} && \text{${\rm div} \, \tilde{U}[\dot{\eta}] =0$ on $\Omega_1$}.
\end{eqnarray}
Consequenlty thanks to \eqref{eq_Ut2}, we deduce classically that $U_h[\dot\eta]$ defined by \eqref{def:Uh} belongs to $H^2(\Omega_h)$.
Moreover, for all $(j,l) \in \mathbb N^2$  s.t. $\max(j,l) \leq k+1,$ we have, because of the $x$-orthogonality of the basis $x \mapsto \displaystyle\exp(in{\pi x\over L}),$ that
\begin{eqnarray*}
\|\partial^{j}_x\partial_z^{l} \Psi ; L^2_{\sharp}((0,L)\times (0,1))\|^2 &\leq& C
\sum_{n \in \mathbb Z}  \int_0^{\frac{1}{n}} |n|^{2(j+l-1)} |\dot{\eta}_n|^{2} |Q^{(l)} (|n|z)|^2 {\rm d}z\\
& \leq & C
\left( \sum_{n \in \mathbb Z}  |n|^{2(j+l)-3} |\dot{\eta}_n|^{2} \right) \int_{0}^1 |Q^{(l)} (z)|^2 {\rm d}z\\
& \leq & C \|\dot{\eta} ; H_{\sharp}^{j+l- \frac 32}(0,L )\|^2 \|Q ; H^{l}(0,1)\|^2\,.
\end{eqnarray*}
Consequently, when $\dot{\eta} \in H^{k -1/2}_{\sharp}(0,L),$ we get that $\tilde{U} [\dot{\eta}]\in H^{k}_{\sharp}(\Omega_1)$ and that  the following estimate is satisfied
\begin{equation} \label{eq_tildeU}
\|\tilde{U}[\dot{\eta}] ; H^{k}_{\sharp}(\Omega_1)\| \leq \|Q ; H^{k+1}(0,1)\| \|\dot{\eta} ; H^{k-\frac 12}_{\sharp}(0,L)\|.  
\end{equation}
Up to the change of variable which depends only on $\lambda$ and is thus bounded by a constant which depends only on $R_0,$ we obtain that $U_{\lambda}$ also satisfies continuity estimates similar to   \eqref{est:Uh:L2}-\eqref{est:Uh:H1}-\eqref{est:Uh:H2}. This ends the proof.
\end{proof}

\begin{rem}
An alternative construction for $U_h$ reads:
$$
U_h[\dot{\eta}] =B_h^{-\top}\begin{pmatrix} -1 & 0 \\ 0 & 1\end{pmatrix}\tilde{U}[\dot{\eta}]\left(x, \dfrac{ h(x)-y}{h(x)} \right)
$$
where $B_h$ is defined by \eqref{matrice:h}. With this construction, we may exploit the better continuity estimates for $\dot{\eta} \mapsto U_h[\dot{\eta}]$ (gain of 1/2 derivative, see \eqref{eq_tildeU}). Yet, this construction requires more regularity of $h$ than the construction given in Propostion \ref{prop_uetoile}.
\end{rem}
 \subsection{Regularity identities}
\begin{lemma} Let $h\in H^2(0, T; L^2(0, L)\cap L^2(0, T, H^4(0, L))$, such that $h^{-1}\in L^\infty((0, T)\times(0,L))$. For any triplet $(w, q, b)$ satisfying the regularity assumptions of Definition \ref{def_strongsolution},  namely
\begin{equation} \label{eq_regulariteb}
b \in H^2(0,T ; L^2_{\sharp}(0, L)) \cap L^2(0,T; H^4_{\sharp}(0, L)), 
\end{equation}
\begin{equation} \label{eq_regularitewq}
w\in H^1_{\sharp}(\mathcal Q_T)\,,  \quad \nabla^2 w \in L^2_{\sharp}(\mathcal Q_T)\,, \quad
q \in L^2_{\sharp}(\mathcal Q_T)\,, \quad \nabla q\in L^2_{\sharp}(\mathcal Q_T)\,,
\end{equation}
and such that $w(x, h(x,t), t)= \partial_t b(x,t) \in L^2_{\sharp,0}(0,L)$ and $\div w=0$, the following identities are satisfied:
\begin{multline}
\label{egalite_vitesse_fluid:app}
\int_{\mathcal Q_t} {\rm div} \sigma(w, q) \cdot   (\partial_t w + \hat{\Lambda} \cdot \nabla w - w \cdot \nabla \hat{\Lambda} )=- \int_0^t\int_{0}^{L} \phi(w,q,h) \partial_{tt} b\\
+\mu  \int_{\mathcal F(t)} |\nabla w|^{2} (t) -\mu  \int_{\mathcal F(0)} |\nabla w|^{2}(0) + 2\mu\int_{\mathcal  Q_t} D(w) :  \left([\nabla \hat{\Lambda}]^{\top} \nabla w + \nabla \hat{\Lambda} [ \nabla w ]^{\top} - D(w\cdot \nabla \hat{\Lambda}) \right),
\end{multline}
and
\begin{multline}
\label{egalite_vitesse_beam:app}
\alpha \int_0^t\int_0^L \partial_{xxxx}b \, \partial_{tt} b - \beta\int_0^t\int_0^L \partial_{xx}b \,\partial_{tt} b -\gamma\int_0^t\int_0^L \partial_{xxt}b\, \partial_{tt} b =\\
 \int_{0}^{L} \dfrac{\gamma}{2} |\partial_{tx} b(x, t)|^2\mathrm{d}x - \beta \partial_{t} b(x,t) \partial_{xx} b(x,t) \mathrm{d}x - \alpha \partial_{tx} b(x,t) \partial_{xxx} b (x,t) \mathrm{d}x
\\
- \beta \int_0^t\int_{0}^{L} |\partial_{tx} b |^2 - \alpha\int_0^t \int_{0}^{L} |\partial_{txx} b |^2
\\ -\int_{0}^{L} \dfrac{\gamma}{2} |\partial_{tx} b(x, 0)|^2 \mathrm{d}x + \beta \partial_t b(x, 0)
\partial_{xx} b(x, 0) \mathrm{d}x   - \alpha \partial_{tx}b(x,0)\partial_{xxx}  b(x,0)\mathrm{d}x  .
\end{multline}
\end{lemma}

\begin{proof}
The second identity, involving $b$ only, is straightforward  since 
applying classical interpolation arguments (see in particular \cite[Theorem 3.1, p. 19]{LM72}) the regularity \eqref{eq_regulariteb} of $b$ implies
that 
\begin{equation} \label{eq_regularitebsupp}
b \in C^0(0, T; H_\sharp^3(0,L))\cap H^1(0,T; H^2_{\sharp}(0,L)) \cap C^1(0, T; H_\sharp^1(0,L))
\end{equation}
We now explain how one can obtain the first identity. Assuming that $(w, b)$ are regular enough we have
\begin{multline*}
\int_{\mathcal Q_t} {\rm div} \sigma(w, q) \cdot   (\partial_t w + \hat{\Lambda} \cdot \nabla w - w \cdot \nabla \hat{\Lambda} )=\\ \int_0^t\int_{y=h(x,s)} \sigma(w,q) n \cdot  (\partial_t w+ \hat{\Lambda} \cdot \nabla w - w \cdot \nabla \hat{\Lambda}) {\rm d}l {\rm d}t- \int_0^t\int_{\mathcal F(s)} \sigma(w,q) : \nabla (\partial_t w + \hat{\Lambda} \cdot \nabla w - w \cdot \nabla \hat{\Lambda})\,.
\end{multline*}
But by the assumption satisfied by $(w, b)$ on the boundary $y=h(x,t)$ and by construction of $\hat{\Lambda}$ 
$$
\begin{array}{rcl}
w \cdot \nabla \hat{\Lambda} = w_2 \partial_2 \Lambda & =& 0 \,, \\
\partial_t w +\hat{\Lambda} \cdot \nabla w - w \cdot \nabla \hat{\Lambda} &= &\partial_{tt} b\,,
\end{array}\quad \text{ on $y=h(x,t)$}\,.
$$
Combining this with the definition \eqref{eq_phi} of $\phi$, we have 
$$
\int_0^t\int_{y=h(x,s)} \sigma (w, q) n \cdot  (\partial_t w + \hat{\Lambda} \cdot \nabla w - w \cdot \nabla \hat{\Lambda}) {\rm d}l{\rm d}t  = - \int_0^t\int_{0}^{L} \phi(w,q,h) \partial_{tt} b \,.
$$

\medskip

On the other hand, we have by construction since $\div {w}=\div \hat \Lambda =0$ 
$$
{\rm div}  \left( \partial_t w + \hat{\Lambda} \cdot \nabla w-w \cdot \nabla \hat{\Lambda}\right) = 0\,.
$$
Consequently, recalling that $\mathcal F(t)$ moves along the characteristics of $\hat{\Lambda},$
we  obtain the following equality
\begin{eqnarray*}
&&\int_{\mathcal Q_t}  \sigma (w, q) : \nabla (\partial_t w + \hat{\Lambda} \cdot \nabla w - w\cdot \nabla \hat{\Lambda})
= 2\mu \int_{\mathcal Q_t} D(w) : D(\partial_t w + \hat{\Lambda} \cdot \nabla w - w\cdot \nabla \hat{\Lambda})) \\
&&= \mu  \int_{\mathcal F(t)} |D w|^{2}  -\mu  \int_{\mathcal F(0)} |D w|^{2}(0) + 2\mu
\int_{\mathcal  Q_t} D(w) : \left(  [\nabla \hat{\Lambda}]^{\top} \nabla w + \nabla \hat{\Lambda} [ \nabla w ]^{\top} - D(w\cdot \nabla \hat{\Lambda})\right)\,.
\end{eqnarray*}
As already noted, thanks to the transverse motion on the fluid domain, to the fact that the trace of $w$ on the moving boundary is colinear to $e_2$  and to the divergence free property of $w$, we have
$$
\int_{\mathcal F (t)} |D w|^{2} = \dfrac{1}{2}\int_{\mathcal F (t)} |\nabla w|^2\,.
$$

Consequently, we deduce that \eqref{egalite_vitesse_fluid:app} holds true for any regular enough couple $(w, b)$.  Even though all the terms in the final identity are well-defined for vector-fields $w$ satisfying \eqref{eq_regularitewq} (we already noticed in Remark \ref{rem:reg_sup} that with the regularity \eqref{eq_regularitewq} we have that $ t \mapsto \int_{\mathcal F(t)}|\nabla w |^2(\cdot,t)$ is continuous) we need here to get more precise on the meaning of boundary terms in integration by parts. Note in particular that, with \eqref{eq_regularitewq}, we only have 
$\partial_t w  \in L^2(\mathcal Q_t) $  so that its trace on $y=h(x,t)$ is not well-defined (though it is divergence-free).

\medskip

There are several ways to overcome this difficulty.
For instance, one may note that the computations above are completely rigorous if we assume further that  $\nabla \partial_t w \in L^2(\mathcal Q_t)$.  We may then end the proof by a density argument.
Indeed, reproducing computations in the proof of {Proposition \ref{prop_cdv}} we have that the regularity \eqref{eq_regulariteb}-\eqref{eq_regularitebsupp}, also satisfied by $h,$ ensures that
the transformation  $w \mapsto \hat{w}$ defined by:
$$
{w}(x,y,s)= B_h^{-\top}(x, \frac{ y}{ h(x,s)}) \, \hat{w}(x,  \frac{y}{ h(x,s)},s)  \quad (x, y,s)\in \mathcal Q_t,
$$
realizes an homeomorphism between the set of $w$ with regularity \eqref{eq_regularitewq}
and the set:
$$
\mathcal U := H^1_{\sharp}(\Omega_1 \times (0,T)) \cap L^2((0,T) ; H^2_{\sharp}(\Omega_1)).
$$
Furthermore this mapping exchanges the subset of $w$ and $\hat{w}$ whose (space)-gradients are square integrable on the time/space domain. As $\mathcal U \subset C([0,T] ; H^1_{\sharp}(\Omega_1)),$ we note also that if $\hat{w}_n$ converges to $\hat{w}$ in  $\mathcal U$ then $t \mapsto \int_{\mathcal{F}(t)} |\nabla w_n(\cdot,t)|^2$ converges to  $t \mapsto \int_{\mathcal{F}(t)} |\nabla w(\cdot,t)|^2$ in $C([0,T]).$
Consequently, we can pass to the limit in the identity \eqref{egalite_vitesse_fluid:app} for a sequence of $w_n$ such that the associated $\hat{w}_n$ converges in $\mathcal U$ to $\hat{w}.$

\medskip

We can then decompose any $\hat{w} \in \mathcal U$ 
into 
$$
\hat{w} = \hat{w}_0 + \tilde{U}[b] \qquad \text{ ( with $\tilde{U}$ defined in the proof of Proposition \ref{prop_uetoile} )}
$$
and approximate $b$ by considering its ($L^{2}$-)orthogonal projection on the first eigenmodes of the beam operator ($\partial_{xxxx}$) and $\hat{w}_0$ by considering its  ($L^2$-)orthogonal projection on the first eigenmodes of the Stokes operator with homogeneous Dirichlet boundary conditions on $z=0$ and $z=1$. 

\end{proof}

\end{document}